\documentclass[12pt]{amsart}
\usepackage{graphicx} 
\usepackage{amssymb,latexsym,eucal}

\usepackage{hyperref}
\usepackage[letterpaper,centering,text={6.5in,9in}]{geometry}
\usepackage{latexsym}
\usepackage[utf8]{inputenc}
\usepackage{amsmath}
\usepackage{amsthm}
\usepackage{bigints}
\usepackage{amsfonts}
\usepackage{amssymb}
\usepackage{epsfig}
\usepackage{nicefrac}
\usepackage[all]{xy}
\usepackage{graphicx}
\usepackage{enumerate}
\usepackage{enumitem}
\usepackage{mathtools}    
\DeclarePairedDelimiterX\setc[2]{\{}{\}}{\,#1 \;\delimsize\vert\; #2\,}
\usepackage{float}
\usepackage{pgfplots}
\usepackage{bbm}
\usepackage{caption,subcaption}
\usepackage{breqn}
\usepackage{amsaddr}
\usepackage{tikz}
\usetikzlibrary{
intersections, arrows.meta,
automata,er,calc,
backgrounds,
mindmap,folding,
patterns,
decorations.markings,
fit,decorations.pathmorphing,
shapes,matrix,
positioning,
shapes.geometric,
arrows,through, 
}
\usepackage{tikz-cd}

\def\bigmid{\ \rule[-3.5mm]{0.1mm}{9mm}\ }

\newtheorem{theorem}{Theorem}[section]
\newtheorem{introtheorem}{Theorem}

\newtheorem{corollary}[theorem]{Corollary}
\newtheorem{proposition}[theorem]{Proposition}

\theoremstyle{definition}

\newtheorem{remark}[theorem]{Remark}

\newtheorem{definition}[theorem]{Definition}

\newcommand{\RR}{{\mathbb R}}
\newcommand{\ZZ}{{\mathbb Z}}

\newcommand{\QQ}{{\mathbb Q}}

\newcommand{\ton}{{\otimes_{\Lambda_{\geq 0}}}}

\newcommand{\lr}{{\,\,\longrightarrow\,\,}}

\title[Barcode entropy and relative symplectic cohomology]{Barcode entropy and relative symplectic cohomology}
\author[Jonghyeon Ahn]{Jonghyeon Ahn}
\newcommand{\Addresses}{{
\bigskip
\bigskip
\footnotesize
\textsc{Institute for Basic Science, Center for Geometry and Physics, Pohang, 37673, South Korea}\par\nopagebreak
\textit{E-mail address}: \texttt{jahn@ibs.re.kr}}}

\date{}

\begin{document}

\maketitle
\begin{abstract}
    In this paper, we study the barcode entropy—the exponential growth rate of the number of not-too-short bars—of the persistence module associated with the relative symplectic cohomology $SH_M(K)$ of a Liouville domain $K$ embedded in a symplectic manifold $M$. Our main result establishes a quantitative link between this Floer-theoretic invariant and the dynamics of the Reeb flow on $\partial K$. More precisely, we show that the barcode entropy of the relative symplectic cohomology $SH_M(K)$ is bounded above by a constant multiple of the topological entropy of the Reeb flow on the boundary of the domain, where the constant depends on the embedding of $K$ into $M$.
\end{abstract}

\vspace{1cm}

\tableofcontents
\section{Introduction}
\subsection{Motivation}
A \textit{persistence module}, which was originally introduced in the framework of topological data analysis to encode the evolution of (co)homological features across filtrations, has found significant applications beyond its initial scope. In particular, this concept was adapted to the realm of symplectic geometry by Polterovich and Shelukhin \cite{ps}, providing a powerful algebraic formalism for studying quantitative and dynamical aspects of Hamiltonian systems. Following their work, persistence modules have become a central object in the investigation of Floer-theoretic invariants and their connections to Hamiltonian dynamics. 

\vspace{0.2cm}

One of the fundamental results in the theory of persistence modules is the \textit{normal form theorem}, which asserts that every persistence module $(V,\pi)$ is completely determined by a multiset $\mathcal{B}(V,\pi)=\{(I_i,m_i)\}$ of intervals $I_i$ with multiplicity $m_i$. This multiset $\mathcal{B}(V,\pi)$ is called the \textit{barcode} of $(V,\pi)$ and the elements of $\mathcal{B}(V,\pi)$ are called \textit{bars}. The \textit{barcode entropy} of a persistence module $(V,\pi)$, introduced by Cineli, Ginzburg, and Gurel \cite{cgg}, measures the exponential growth rate of the “not-too-short” bars in the barcode $\mathcal{B}(V,\pi)$. While the barcode entropy has been studied in several Floer-theoretic settings, its behavior in relative symplectic cohomology setting has remained unexplored. In this paper, we initiate such a study by introducing and analyzing the barcode entropy associated with the relative symplectic cohomology introduced by Varolgunes in \cite{v}.

\vspace{0.2cm}

The initial investigation of barcode entropy by Cineli, Ginzburg, and Gurel \cite{cgg} focused on the persistence module structure arising from the Lagrangian Floer cohomology of pairs of Hamiltonian isotopic Lagrangian submanifolds. For two Lagrangian submanifolds $L$ and $L'$ with $L' = \varphi_H(L)$ where $\varphi_H$ is the Hamiltonian diffeomorphism corresponding to $H$, we denote the barcode entropy of $HF(L,L')$ by $\hbar(\varphi_H;L)$. Their central result is that this barcode entropy is controlled by the dynamics of $\varphi_H$: They proved that the barcode entropy $\hbar(\varphi_H;L)$ is bounded above by the topological entropy $h_{\textrm{top}}(\varphi_H)$ of the Hamiltonian diffeomorphism $\varphi_H$.

\vspace{0.2cm}

This construction also yields a notion of barcode entropy $\hbar(\varphi_H)$ for the diffeomorphism $\varphi_H$ itself by considering the diagonal $\Delta$ and the graph $\Gamma_{\varphi_H}$ of $\varphi_H$, namely,
\begin{align*}
    \hbar(\varphi_H) = \hbar(\textrm{id}\times \varphi_H;\Delta)
\end{align*}
In this case, we still have the same upper bound:
\begin{align*}
    \hbar(\varphi_H) \leq h_{\textrm{top}}(\varphi_H).
\end{align*}
Notably, these barcode entropies are finite because the topological entropy of a smooth map is known to be finite. Moreover, they established a nontrivial lower bound for $\hbar(\varphi_H)$, showing that
\begin{align*}
    \hbar(\varphi_H) \geq h_{\textrm{top}}(\varphi_H|_X)
\end{align*}
whenever $X$ is a compact \textit{hyperbolic} invariant set of $\varphi_H$. This lower bound, in particular, shows that $\hbar(\varphi_H)$ is not always zero. As a consequence, when the ambient symplectic manifold is two-dimensional, the following equality holds:
\begin{align*}
    \hbar(\varphi_H) =  h_{\textrm{top}}(\varphi_H).
\end{align*}

\vspace{0.1cm}
Following the introduction of barcode entropy in \cite{cgg}, a variety of related notions have been developed in different Floer-theoretic and dynamical settings. A Morse-theoretic notion of barcode entropy was introduced in \cite{ggm}, where it was shown to be comparable to the topological entropy of the geodesic flow. The framework was extended to symplectic cohomology of Liouville domains in \cite{fls} and the barcode entropy was shown to admit an upper bound in terms of the topological entropy of the Reeb flow on the boundary. This result was later complemented by a matching lower bound in \cite{cggm}. In the original work \cite{cgg}, no lower bound was obtained for the Lagrangian barcode entropy $\hbar(\varphi_H;L)$. Such a lower bound was established later in \cite{m}, further clarifying the dynamical significance of this invariant. More recently, analogous upper and lower bounds for barcode entropy have been obtained in the setting of wrapped Floer cohomology; see \cite{f,f2}.

\vspace{0.2 cm}

Despite these developments, existing results on barcode entropy have primarily focused on absolute Floer-theoretic invariants. Relative symplectic cohomology occupies a fundamentally different position: it captures Floer-theoretic information that depends not only on the intrinsic geometry of a Liouville domain but also on how the domain is situated inside a larger symplectic manifold. This additional flexibility raises new conceptual questions about how persistence and entropy should reflect both ambient and boundary dynamics. The goal of this paper is to address this gap by developing a notion of barcode entropy for relative symplectic cohomology and by establishing a precise quantitative relation with the dynamics of the Reeb flow on the boundary.

\subsection{Main results}
In this subsection, we state our main results. As indicated above, we will introduce the barcode entropy of the relative symplectic cohomology $SH_M(K)$. Our primary contribution is to establish an upper bound for this barcode entropy: it is bounded above by a constant multiple of the topological entropy of the Reeb flow. The constant depends on how the subset $K$ sits inside the ambient symplectic manifold $M$, reflecting the genuinely relative nature of the invariant.  

\vspace{0.2cm}

Let $(M,\omega)$ be a closed symplectic manifold and $K\subset M$ be a compact subset. Varolgunes \cite{v} introduced the relative symplectic cohomology $SH_M(K)$ of $K$ in $M$. This relative symplectic cohomology admits a persistence module structure under a suitable condition, which we now explain. A symplectic manifold $(M,\omega)$ is called \textit{symplectically aspherical} if $\omega|_{\pi_2(M)}=0$ and $c_1(TM)|_{\pi_2(M)}=0$ where $c_1(TM)$ is the first Chern class of the tangent bundle $TM$ of $M$. For our purposes, the subset $K \subset M$ is required to be a \textit{Liouville domain}. This means that $K$ is a submanifold of $M$ of codimension 0 with smooth boundary and there exists a vector field $X$ on $K$ such that $\mathcal{L}_X \omega = \omega$. A key feature of a Liouville domain is that its boundary $\partial K$ carries a natural contact structure: the contact form $\alpha$ is given by the contraction of the symplectic form $\omega$ with the vector field $X$, namely, $\alpha = \iota_X\omega|_{\partial K}$. For such a Liouville domain $K \subset M$, we say that the contact manifold $(\partial K, \alpha)$ is \textit{index-bounded} if for each $k\in\ZZ$, the set of periods of Reeb orbits on $(\partial K, \alpha)$ that are contractible in $M$ and have Conley-Zehnder index $k$ of $(\partial K, \alpha)$ is bounded. 

\vspace{0.2cm}

Following the definition of the barcode entropy of a persistence module given in \cite{cgg}, we can define the \textbf{relative symplectic cohomology barcode entropy} and we denote it by $\hbar(SH_M(K))$. Informally, this quantity measures the exponential growth rate of the ``not-too-short" bars in the barcode of the persistence module $SH_M(K)$: For $\epsilon>0$, let $$b_\epsilon(\textrm{tru}(SH_M(K)),\sigma)$$ be the number of bars of the truncation of the persistence module $SH_M(K)$ at $\sigma$ whose lengths are greater than or equal to $\epsilon$. Note that $b_\epsilon(\textrm{tru}(SH_M(K)),\sigma)$ is an increasing function of $\sigma$ and a decreasing function of $\epsilon$. We define
\begin{align*}
    \hbar_\epsilon(SH_M(K)) = \limsup_{\sigma \to \infty} \frac{1}{\sigma}\log^+b_\epsilon(\textrm{tru}(SH_M(K)),\sigma),
\end{align*}
and
\begin{align*}
    \hbar(SH_M(K)) = \lim_{\epsilon \to 0} \hbar_\epsilon(SH_M(K))
\end{align*}
where $\log^+a = \max\{\log_2 a, 0\}$.
\vspace{0.2cm}

Our result is that the barcode entropy $\hbar(SH_M(K))$ is bounded above by a constant multiple of the topological entropy $h_{\textrm{top}}(\varphi_\alpha^t)$ of the Reeb flow $\varphi_\alpha^t$ on $(\partial K, \alpha)$, where the constant depends on the embedding of $K$ inside $M$.

\begin{introtheorem}[Theorem \ref{thmare}]\label{thma}
    Let $(M,\omega)$ be a closed symplectic manifold and  $K \subset M$ be a Liouville domain. Assume that $(M,\omega)$ is symplectically aspherical and $(\partial K, \alpha)$ is index-bounded. Then there exists a constant $C=C(M,K) >0$, depending on the pair $(M,K)$, such that
    \begin{align*}
        \hbar(SH_M(K)) \leq C \,h_{\textrm{top}}(\varphi_\alpha^1)
    \end{align*}
    where $\varphi_\alpha^t$ is the Reeb flow on $(\partial K, \alpha)$.
\end{introtheorem}

The constant $C(M,K)$ in Theorem~\ref{thma} arises from the choice of a collar neighborhood of $\partial K$ symplectically embedded in $M$. Fixing such a collar of finite length, one considers Hamiltonian functions that are linear with respect to the radial coordinate on a subregion of the collar. The estimate in Theorem~\ref{thma} involves the infimum of the slopes of such Hamiltonian functions, and the optimal bound is obtained by varying the collar length up to its maximal value allowed by the embedding. Consequently, the constant $C(M,K)$ reflects the quantitative size of the embedding of $K$ into $M$. Conceptually, this shows that the complexity of the persistence structure of relative symplectic cohomology is controlled by the dynamical complexity of the Reeb flow on the boundary of the Liouville domain. Unlike the symplectic cohomology of an isolated Liouville domain, the relative invariant retains information about the surrounding symplectic geometry near the boundary, highlighting its sensitivity to the chosen embedding of $K$ into $M$.

\vspace{0.2cm}

An immediate consequence of Theorem \ref{thma} is that the barcode entropy $\hbar(SH_M(K))$ is finite because the topological entropy of a smooth flow is known to be finite. Moreover, the growth of the number of not-too-short bars in $SH_M(K)$ is at most exponential.

\vspace{0.2cm}

\noindent
\textbf{Organization of the paper.} 
Since the material presented in this paper encompasses several fields of mathematics, we begin in Section 2 by collecting the basic notions and facts from each subject that will be used throughout the paper. This preliminary section is not intended to be comprehensive; readers are encouraged to consult the references provided in each subsection for further details. In Section 3, we elaborate on the definition and structure of the relative symplectic cohomology of a Liouville domain. In Section 4, we formally define the relative symplectic cohomology barcode entropy, $\hbar(SH_M(K))$, and explore its relation to the previously established barcode entropy associated with a Hamiltonian function. The proof of the upper bound, Theorem \ref{thma}, is carried out in Section 5, where we first explain the key concepts and ingredients required for the argument.

\vspace{0.2cm}

\noindent
\textbf{Acknowledgment.} The author thanks Rafael Fernandes and Joao Pering for helpful discussions, and is grateful to Yong-Geun Oh for his encouragement and insightful comments. The author also thanks Ely Kerman for reading an early draft of this paper.

\section{Preliminaries}
\subsection{From persistence module theory}
A concise review of persistence module theory will be presented below, with most of the exposition adapted from \cite{csgo, prsz}. Before giving the definition of persistence module, let us fix a ground field $\mathbbm{k}$ and we usually suppress this from our notation.

\begin{definition}
A \textbf{persistence module} is a pair $(V, \pi)$ where $V$ is a collection $\{V_\tau\}_{\tau\ \in\RR}$ of finite dimensional vector spaces over $\mathbbm{k}$ and $\pi = \{\pi_{\tau_1\tau_2} : V_{\tau_1} \lr V_{\tau_2}\}_{\tau_1, \tau_2 \in \RR}$ is a collection of linear maps from $V_{\tau_1}$ to $V_{\tau_2}$ for $\tau_1\leq \tau_2$ satisfying the following:
    \begin{enumerate}[label=(\alph*)]
        \item For any $\tau \in \RR$, $\pi_{\tau\tau} = id_{V_\tau}$ and for any $\tau_1 \leq \tau_2 \leq \tau_3$, $\pi_{\tau_1 \tau_3} = \pi_{\tau_2 \tau_3} \circ \pi_{\tau_1 \tau_2}$.
        \begin{align*}
            \includegraphics[]{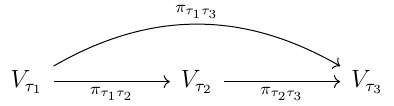}
        \end{align*}
        \item There exists a closed, bounded from below and nowhere dense subset $\textrm{Spec}(V, \pi) \subset \RR$, which is called the \textbf{spectrum} of $(V, \pi)$ such that $\pi_{\tau_1 \tau_2} : V_{\tau_1} \lr V_{\tau_2}$ is an isomorphism whenever $\tau_1$ and $\tau_2$ are in the same connected component of $\RR \setminus \textrm{Spec}(V,\pi)$.
        \item For any $\tau \in \RR$ and any $\tau' \leq \tau$ sufficiently close to $\tau$, the map $\pi_{\tau' \tau} : V_{\tau'} \lr V_\tau$ is an isomorphism.
        \item There exists $\tau_0 \in \RR$ such that $V_\tau = 0$ for $\tau\leq \tau_0$.
 \end{enumerate}
    We will usually use the notation $V$ to denote the persistence module $(V,\pi)$ when it is clear from the context.
\end{definition}

One crucial example of a persistence module is an interval persistence module. For an interval $I = (a,b], $ where $-\infty < a < b \leq\infty$, the \textbf{interval persistence module} $\mathbbm{k}I = (V, \pi)$ is defined by
\begin{align*}
    V_\tau = \begin{cases}
        \mathbbm{k} & \text{if}\,\,\tau \in I \\
        0 & \text{if} \,\, \tau\notin I,
    \end{cases}
\end{align*}
and 
\begin{align*}
    \pi_{\tau_1 \tau_2} = \begin{cases}
        \textrm{id}_{\mathbbm{k}} & \text{if}\,\,\tau_1,\tau_2 \in I\\
        0 &\text{otherwise.}
    \end{cases}
\end{align*}

\vspace{0.2 cm}
A key fact in the theory of persistence module is the \textit{normal form theorem}. Before stating it, we need the notion of equivalence between persistence modules. Let $(V, \pi)$ and $(V',\pi')$ be two persistence modules. A \textbf{morphism} $A : (V,\pi) \lr (V',\pi')$ is a family of linear maps $\{ A_\tau : V_\tau \lr V'_\tau\}_{\tau \in \RR}$ such that $A_{\tau_2} \circ \pi_{\tau_1 \tau_2} = \pi'_{\tau_1 \tau_2} \circ A_{\tau_1}$ for all $\sigma \leq \tau$.
        \begin{align*}
            \includegraphics[]{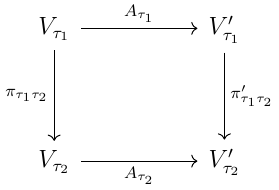}
        \end{align*}
Two persistence modules $(V, \pi)$ and $(V',\pi')$ are said to be \textbf{isomorphic} if  there exist two morphisms
\begin{align*}
    A : (V,\pi) \lr (V',\pi')\,\,\textrm{and}\,\,B : (V',\pi') \lr (V,\pi)
\end{align*}
such that $A \circ B : (V', \pi') \lr (V', \pi')$ is the identity morphism and $B \circ A : (V, \pi) \lr (V, \pi)$ is also the identity morphism. Also, we can define the \textbf{direct sum} $(V \oplus V',\pi\oplus\pi')$ of two persistence module $V$ and $V'$ by
\begin{align*}
    (V \oplus V')_\tau = V_\tau \oplus V'_\tau\,\,\textrm{and}\,\, (\pi\oplus\pi')_{\tau_1 \tau_2} = \pi_{\tau_1 \tau_2} \oplus \pi'_{\tau_1 \tau_2}
\end{align*}
and this definition extends in a straightforward manner to infinitely many persistence modules. 
\begin{theorem}[Normal form theorem]\label{nft}
    Let $(V, \pi)$ be a persistence module. Then there exists a countable collection $\{ (I_i, m_i)\}_{i=1,2,3,\cdots}$ of intervals $I_i$ with multiplicity $m_i$ such that
    \begin{align*}
        (V, \pi) \cong \bigoplus_{i=1}^\infty \left(\mathbbm{k} I_i\right)^{m_i}.
    \end{align*}
    Moreover, this decomposition is unique up to a permutation, that is, to any persistence module $(V,\pi)$, there exists a unique multiset $\mathcal{B}(V,\pi)$ consisting of intervals $I_i$ with multiplicity $m_i$. This multiset $\mathcal{B}(V,\pi)$ is called the \textbf{barcode} of the persistence module $V$ and an element of $\mathcal{B}(V,\pi)$ is called a \textbf{bar} of $(V,\pi)$.
\end{theorem}

We end this subsection by introducing a useful operation on persistence modules which will play an important role in the sequel. Let $\sigma\in \RR$. The \textbf{truncation} $\text{tru}(V,\pi ;\sigma ) = (\text{tru}(V;\sigma), \text{tru}(\pi;\sigma))$ of $(V,\pi)$ is defined by
        \begin{align*}
            \text{tru}(V;T)_\tau = \begin{cases}
                V_\tau &\text{if}\,\, \tau<\sigma\\
                0 &\text{otherwise}
            \end{cases}
         \end{align*}
        and
        \begin{align*}
            \text{tru}(\pi;\sigma)_{\tau_1 \tau_2} = \begin{cases}
                \pi_{\tau_1 \tau_2}&\text{if}\,\, \tau_2< \sigma\\
                0 &\text{otherwise}.
            \end{cases}
        \end{align*}
\vspace{0.3 cm}
\subsection{From Floer theory}\label{floer} This subsection reviews the fundamental aspects of Floer theory required in the sequel. Experienced readers may wish to proceed directly to the next section.

\subsubsection{Hamiltonian Floer cohomology}
We begin by introducing the ground ring. The \textbf{Novikov field} $\Lambda$ is defined by
\begin{align*}
    \Lambda = \left\{ \sum_{i=1}^\infty a_i T^{\lambda_i} \bigmid a_i \in \QQ, \lambda_i \in \RR\,\,\text{and}\,\, \lim_{i \to \infty} \lambda_i = \infty \right\} 
\end{align*}
where $T$ is a formal variable. There is a \textbf{valuation map} $\textrm{val} : \Lambda \to \RR \cup \{\infty\}$ given by
\begin{align*}
    \textrm{val} (x) = 
    \begin{cases}
      \displaystyle\min_{i} \{\lambda_i \mid a_i \neq 0 \} \,&\text{if}\,\, x = \displaystyle\sum_{i=1}^\infty a_i T^{\lambda_i} \neq 0\\
      \infty \,&\text{if}\,\, x = 0.
    \end{cases}   
\end{align*}
For any $r \in \RR$, define $\Lambda_{\geq r} = \textrm{val}^{-1}([r,\infty])$. In particular, we call
\begin{align*}
    \Lambda_{\geq 0} = \left\{ \sum_{i=1}^\infty a_i T^{\lambda_i} \in \Lambda \bigmid  \lambda_i \geq 0 \right\}
\end{align*}
the \textbf{Novikov ring}. 
\vspace{0.2 cm}

\par Let $(M, \omega)$ be a symplectic manifold and let $H : S^1 \times M \to \RR$ be a Hamiltonian function on $M$. The \textbf{Hamiltonian vector field} $X_H$ of H is defined by 
\begin{align*}
    \iota_{X_H} \omega = dH.
\end{align*}
We say a Hamiltonian function $H :S^1 \times M \to \RR$ is \textbf{nondegenerate} if every 1-periodic orbit $x$ of $X_H$ is nondegenerate, that is, for the Hamiltonian flow $\phi_H^t$ of $H$, the map $$(d \phi^1_H)_{x(0)} : T_{x(0)}M \lr T_{x(1)}M = T_{x(0)}M$$ has no eigenvalue equal to 1. We denote the set of all nondegenerate contractible 1-periodic orbits of $X_H$ by $\mathcal{P}(H)$. Let $x \in \mathcal{P}(H)$ and $c_x$ be a disk capping. The \textbf{action} of $(x, c_x)$ is defined by
\begin{align}\label{act}
    \mathcal{A}_H(x, c_x) = \int_{D^2} c_x^* \omega + \int_{S^1} H(t, x(t))dt.
\end{align}
We can associate to a pair $(x, c_x)$ an integer using the \textit{Conley-Zehnder index}, which we denote by $\mu_{\text{CZ}}(x,c_x)$. Define an equivalence relation on the set of pairs of a contractible 1-periodic orbit and its capping by
\begin{align*}
    (x, c_x) \sim (y, c_y) \,\,\text{if and only if}\,\, x = y, \,\mathcal{A}_H(x, c_x) = \mathcal{A}_H(y, c_y) \,\,\text{and}\,\, \mu_{\text{CZ}}(x,c_x) = \mu_{\text{CZ}}(y,c_y).
\end{align*}
We denote the equivalence class of $(x, c_x)$ by $[x,c_x]$ and denote the set of all equivalence classes by $\widetilde{\mathcal{P}}(H)$. The \textbf{Floer complex of $H$} is defined by
\begin{align*}
    CF(H) = \left\{\sum_{i=1}^{\infty}a_i [x_i, c_{x_i}] \bigmid a_i \in \QQ,  [x_i, c_{x_i}] \in  \widetilde{\mathcal{P}}(H)\,\,\text{and}\,\, \lim_{i \to \infty} \mathcal{A}_H([x_i, c_{x_i}]) = \infty \right\}.
  \end{align*}
The grading of $CF(H)$ is given by the \textit{Conley-Zehnder index} (See \cite{sz}), that is,
\begin{align*}
    | [x, c_x]| = \mu_{\text{CZ}}([x, c_x]).
\end{align*}
In particular, the $k$-th Floer complex $CF^k(H)$ of $H$ is given by
\begin{align*}
   CF^k(H)= \left\{\sum_{i=1}^{\infty}a_i [x_i, c_{x_i}] \in CF(H) \bigmid \mu_{\text{CZ}}([x, c_x])=k  \right\}.
\end{align*}
Fix a generic almost complex structure $J$ on the tangent bundle $TM$ of $M$. For $x, y \in \mathcal{P}(H)$, let $\pi_2(M, x, y)$ be the set of homotopy classes of smooth maps from $\RR \times S^1$ to $M$ which are asymptotic to $x$ and $y$ at $-\infty$ and $\infty$, respectively. For $[x, c_x], [y, c_y] \in \widetilde{\mathcal{P}}(H)$ and $A \in \pi_2(M,x,y)$, the moduli space $\mathcal{M}(H,J; [x, c_x], [y, c_y]; A)$ is the set of smooth maps $u : \RR \times S^1 \lr M$ satisfying the following. \vspace{0.2cm}
\begin{itemize}
    \item The homotopy class of $u$ represents $A \in \pi_2(M,x, y)$.\\
    \item $\displaystyle\frac{\partial u}{\partial s} + J(u) \left( \frac{\partial u}{\partial t} - X_{H}(u) \right) = 0$.\\
    \item $\displaystyle\lim_{s \to -\infty} u(s,t) = x(t)$ and $\displaystyle\lim_{s \to \infty} u(s,t) = y(t)$.\\
    \item $[y, c_y] = [y, c_x\,\#\, (-u)] = [y, c_x\,\#\, (-A)] $ where $\#$ denotes the connected sum.\vspace{0.2cm}
\end{itemize}
Note that if $u \in \mathcal{M}(H,J; [x, c_x], [y, c_y]; A)$, then $ \mathcal{A}_H([x, c_x]) \leq \mathcal{A}_H([y,c_y])$ and $\mu_{\text{CZ}}([x, c_x]) \leq \mu_{\text{CZ}}([y, c_y])$. The \textbf{Floer differential} $d : CF^*(H) \lr CF^{*+1}(H)$ of the Floer complex $CF(H)$ is given by
\begin{align*}
    d [x, c_x] = \sum_{\substack{[y, c_y] \in \widetilde{\mathcal{P}}(H) \\ A \in \pi_2(M, x, y)}} \#\mathcal{M}(H,J ; [x, c_x], [y, c_y]; A) T^{E_{\text{top}}(u)} [y, c_y]
\end{align*}
where $\#\mathcal{M}(H ; [x, c_x], [y, c_y]; A)$ is the virtual count of the moduli space $\mathcal{M}(H ; [x, c_x], [y, c_y]; A)$ (see \cite{p}) and the \textbf{topological energy} $E_{\text{top}}(u)$ of a Floer trajectory $u$ is given by
\begin{align*}
    E_{\text{top}}(u) = \int_{S^1} H_t(y(t)) dt - \int_{S^1} H_t(x(t)) dt + \omega(A).
\end{align*}
It can be easily shown that the topological energy is given by the difference of actions of two capped orbits connected by a Floer trajectory: $$E_{\text{top}}(u) = \mathcal{A}_H([y,c_y]) - \mathcal{A}_H([x, c_x]).$$ Since the Floer differential $d$ satisfies $d \circ d = 0$, the \textbf{Floer cohomology} $HF(H)$ of $H$ is defined to be 
\begin{align*}
    HF^*(H) = H^* \left( CF(H), d \right).
\end{align*}

\vspace{0.2 cm}

The Floer cochain complex admits a filtration by the action functional. For each $L \in \RR\cup \{-\infty\}$, define
\begin{align*}
    CF^{> L}(H) = \left\{ \sum_{i=1}^{\infty}a_i [x_i, c_{x_i}] \in CF(H) \bigmid \mathcal{A}_H([x_i, c_{x_i}]) > L\,\,\text{for all}\,\,i \right\}.
\end{align*}
This subset $CF^{> L}(H)$ of $CF(H)$ is actually a subcomplex of $CF(H)$ because the Floer differential increases the action. Therefore, we can define
\begin{align*}
    HF^{> L, *} (H) = H^* (CF^{> L}(H), d).
\end{align*}
For $L<L'$, we further define the truncated Floer cohomology
\begin{align*}
    HF^{(L,L'],*}(H) = H^*\left( CF^{>L}(H) / CF^{>L'}(H), \overline{d} \right)
\end{align*}
where $\overline{d}$ is the differential of $CF^{>L}(H) / CF^{>L'}(H)$ induced by $d$ on the quotient complex..

\vspace{0.2 cm}

For two Hamiltonian functions $H_1$ and $H_2$ with $H_1 \leq H_2$, choose a monotone homotopy $H_s$ of Hamiltonian functions from $H_1$ to $H_2$. Then there exists a \textbf{continuation map} 
\begin{align*}
   c^{H_1, H_2} : CF(H_1) \lr CF(H_2)
\end{align*}
defined by the suitable count of Floer trajectories of $H_s$ connecting the orbits of $H_1$ and $H_2$. This  map is also weighted by $T^{E_{\text{top}}(u)}$ in a similar way that we define the Floer differential above. Every continuation map induces a map $HF(H_1) \lr HF(H_2)$. Note that continuation maps increase the actions and respect the gradings, and hence the map
\begin{align*}
    c^{H_1, H_2} : CF^{>L}(H_1) \lr CF^{>L}(H_2)
\end{align*}
makes sense and this also induces a map $HF^{>L}(H_1) \lr HF^{>L}(H_2)$. Of course, the continuation map induces a map $HF^{(L,L']}(H_1) \lr HF^{(L,L']}(H_2)$.

\vspace{0.2 cm}

The Floer cohomology $HF(H;\Lambda)$ of $H$ with coefficient in $\Lambda$ is defined by
\begin{align*}
    HF^*(H;\Lambda) = H^*\left( CF(H) \ton \Lambda \right) = HF^*(H)\ton \Lambda.
\end{align*}

\vspace{0.2 cm}

\subsubsection{Relative symplectic cohomology}
We briefly recall the relative symplectic cohomology introduced by Varolgunes in \cite{v}. Detailed explanation and its application can also be found in \cite{a, a2, dgpz, msv, sun}. 

\vspace{0.2 cm}

A major ingredient of relative symplectic cohomology is to complete the Floer complex, which we will explain now. For any $\Lambda_{\geq0}$-module $A$, we can define its completion as follows: For $r' >r$, there exists a map 
\begin{align*}
    \Lambda_{\geq 0} / \Lambda_{\geq r'} \lr \Lambda_{\geq 0} / \Lambda_{\geq r} 
\end{align*}
and this map induces a map
\begin{align}\label{inverse}
    A \ton \Lambda_{\geq 0} / \Lambda_{\geq r'} \lr A \ton \Lambda_{\geq 0} / \Lambda_{\geq r}.
\end{align}
Then $\left\{ A \ton \Lambda_{\geq 0} / \Lambda_{\geq r} \right\}_{r>0}$ is an inverse system with the maps given by \eqref{inverse}. The \textbf{completion} $\widehat{A}$ of $A$ is defined by
\begin{align*}
    \widehat{A} = \varprojlim_{r \to 0} A \otimes \Lambda_{\geq 0} / \Lambda_{\geq r}.
\end{align*}

\vspace{0.2 cm}

Let $(M, \omega)$ be a closed symplectic manifold and let $K \subset M$ be a compact subset. We say that a nondegenerate Hamiltonian function $H$ is \textbf{K-admissible} if $H$ is negative on $S^1 \times K$. We denote the set of all $K$-admissible Hamiltonian functions by $\mathcal{H}_K$. We can define the \textbf{relative symplectic cohomology} $SH_M(K)$ of $K$ in $M$ by
\begin{align}\label{reldef}
    SH^*_M(K) = H^* \left( \widehat{\varinjlim_{H \in \mathcal{H}_K}} CF(H) \right)
\end{align}
where $\displaystyle\widehat{\varinjlim_{H \in \mathcal{H}_K}} CF(H)$ denotes the completion of $\displaystyle\varinjlim_{H \in \mathcal{H}_K} CF(H)$. Also, define 
\begin{align*}
    SH^*_M(K; \Lambda) = SH^*_M(K)\ton \Lambda.
\end{align*}

\vspace{0.3 cm}

\subsection{From dynamical systems theory} 
In this subsection, we provide a summary of some facts from the theory of dynamical systems that we need throughout this paper. For a comprehensive exposition, see \cite{kh, th}.

\vspace{0.2cm}

Let $(X,d)$ be a compact metric space and let $f : X \lr X$ be a continuous map. Define an increasing sequence $\{d_\ell^f\}$ of metrics for $\ell=1,2,3,\cdots$ by
\begin{align}\label{metric}
    d_\ell^f(x,y) = \max_{1 \leq i\leq \ell-1} d(f^i(x),f^i(y))
\end{align}
where $f^i$ denotes the composition of $f$ with itself $i$ times. Let
\begin{align*}
    B_f(x,\epsilon,\ell) = \left\{y \in X \,\bigmid \, d_\ell^f(x,y) < \epsilon \right\}.
\end{align*}
A set $E \subset X$ is said to be \textbf{$(\ell,\epsilon)$-spanning} if 
\begin{align*}
    X \subset \bigcup_{x \in E} B_f(x,\epsilon,\ell).
\end{align*}
Let $S_d(f,\epsilon,\ell)$ be the minimal cardinality of an $(\ell,\epsilon)$-spanning set. Then the \textbf{topological entropy} $h_{\textrm{top}}(f)$ of $f$ is defined by
\begin{align*}
   h_{\textrm{top}}(f) = \lim_{\epsilon \to 0} \limsup_{\ell \to \infty} \frac{1}{\ell} \log S_d(f,\epsilon,\ell).
\end{align*}
Obviously, the topological entropy of the identity map is zero. The definition of the topological entropy for a flow is completely parallel. Let $F=\{f^t\}_{t \in \RR }$ be a continuous flow on $X$. The counterpart of the metric \eqref{metric} is the family of metrics
\begin{align*}
    d_T^{F} (x,y) = \max_{0\leq t \leq T} d(f^t(x),f^t(y))
\end{align*}
for each $T \in \RR$. We can define the topological entropy $h_{\textrm{top}}(F) = h_{\textrm{top}}(f^t)$ of the flow $F$ by
\begin{align*}
    h_{\textrm{top}}(F) = \lim_{\epsilon \to 0} \limsup_{T \to \infty} \frac{1}{T} \log S_d(F,\epsilon,T).
    \end{align*}
    Note that the topological entropies of a map and a flow do not depend on the choice of the metric $d$. Some useful properties of topological entropy are collected below.\vspace{0.2cm}
    \begin{itemize}
        \item $h_{\textrm{top}}(F) = h_{\textrm{top}}(f^1)$.\vspace{0.2cm}
        \item For each $t\in \RR$, $h_{\textrm{top}}(f^t) = |t| h_{\textrm{top}}(f^1)$.\vspace{0.2cm}
        \item For each $m \in \ZZ$ and $t \in \RR$, $h_{\textrm{top}}(f^{mt}) = |m| h_{\textrm{top}}(f^t)$.\vspace{0.2cm}
        \end{itemize}
These properties indicate that the topological entropy of a flow $\{f^t\}_{t \in \RR}$ is controlled by its time-1 map $f^1$.

\vspace{0.2cm}

In fact, we will not use the definition of the topological entropy directly in this paper; instead, we will use the volume growth rate. Let $X$ be a smooth Riemannian manifold and $f : X \lr X$ be a smooth map. Denote the graph of $f$ by $\Gamma_f$, that is,
\begin{align*}
    \Gamma_f = \left\{(x,f(x))\in X \times X\bigmid x\in X\right\}.
\end{align*}
For a smooth submanifold $Y$ of $X$, the exponential volume growth rate $h_{\textrm{vol}}(f|_Y)$ is defined to be
\begin{align*}
   h_{\textrm{vol}}(f|_Y) = \limsup_{\ell \to \infty} \frac{1}{\ell} \log  \textrm{Vol}\left(\Gamma_{f^\ell|_Y} \right)
\end{align*}
  where $\textrm{Vol}$ denotes the volume induced by the given Riemannian metric. We have two exponential growth rates associated with $f$; the topological entropy $h_{\textrm{top}}(f|_Y)$ and the volume growth rate $h_{\textrm{vol}}(f|_Y)$. These quantities are related as follows by Yomdin's theorem.
  \begin{theorem}[Yomdin]\label{y}
      Let $X$ be a smooth Riemannian manifold and $f : X \lr X$ be a smooth map. For a smooth submanifold $Y$ of $X$, we have
      \begin{align}\label{yi}
          h_{\textrm{vol}}(f|_Y) \leq h_{\textrm{top}}(f|_Y).
      \end{align}
    \end{theorem}
    \begin{proof}
        See \cite{g} and \cite{y}.\\
    \end{proof}
\begin{remark}
    In Theorem \ref{y} (and throughout this paper), the smoothness of $f$ is understood to mean $C^\infty$-smoothness. If the map $f$ is assumed to be merely $C^r$-smooth for some $r< \infty$, then an additional term, depending on $r$, must be added on the right-hand side of the inequality \eqref{yi}. See \cite{g} and \cite{y} for further details.
\end{remark}    
\vspace{0.2cm}  
\section{Relative symplectic cohomology of a Liouville domain}
This section is devoted to additional properties of relative symplectic cohomology, focusing on the situation where the subset $K$ is a Liouville domain. 
\subsection{Reformulation of the definition}
A compact symplectic manifold $K$ has a \textbf{contact type boundary} if exists an outward pointing vector field $X$ defined on a neighborhood of $\partial K$ satisfying $\mathcal{L}_X \omega = \omega$.  We call the vector field $X$ a \textbf{Liouville vector field}. A fundamental property of a contact type boundary $\partial K$ is that it carries a natural contact structure whose contact form $\alpha$ is given by
\begin{align*}
    \alpha = \iota_X \omega|_{\partial K}.
\end{align*}
If the Liouville vector field $X$ is defined on $K$, then we call $K$ a \textbf{Liouville domain}. Denote by $\widehat{K}$ the \textbf{symplectic completion} of $K$, that is,
\begin{align*}
    \widehat{K} = K \cup_{\partial K} \left( \partial K \times [1, \infty) \right)
\end{align*}
with the symplectic form $\widehat{\omega}$ on $\partial K \times [1,\infty)$ is given by $\widehat{\omega} = d(r \alpha)$ where $\alpha $ is the contact form on $\partial K$ and $r$ is the coordinate on $[1,\infty)$. We will denote the Reeb vector field and the Reeb flow on the contact manifold $(\partial K, \alpha)$ by $R_\alpha$ and $\varphi_\alpha^t$, respectively.

\vspace{0.2 cm}

Let $(M, \omega)$ be a closed symplectic manifold and $K \subset M$ be a compact domain\footnote{By a domain, we mean a submanifold of codimension zero with smooth boundary.} with contact type boundary. Although the symplectic completion $\widehat{K}$ of $K$ cannot be embedded inside of $M$, we can still consider a small collar neighborhood $\partial K \times [1,1+3r_1]$ of $\partial K$ in $M$ for some $r_1 >0$. We say that a Hamiltonian function $H : S^1 \times M \to \RR$ is \textbf{contact type $K$-admissible} if it satisfies the following.\vspace{0.2 cm}
\begin{enumerate}[label=(A\arabic*)]
    \item $H$ is negative on $S^1 \times K$.\vspace{0.2 cm}
    \item $H(t, p, r) =  h(r)$ on $S^1 \times \left(\partial K \times [1, 1+3r_1]\right)$ such that\vspace{0.3 cm}
    \begin{itemize}
        \item $ h(r)=h_1(r)$ on $S^1 \times \left(\partial K \times [1, 1+r_1]\right)$ for some strictly convex and increasing function $h_1$,\vspace{0.2 cm}
        \item $h(r) = s_Hr - i_H$ on $S^1 \times (\partial K \times [1+r_1, 1+2r_1])$ where the slope $s_H$ of $H$ satisfies that $s_H \notin \text{Spec}(\partial K, \alpha)$ and $i_H \in \RR$. Here, $\text{Spec}(\partial K, \alpha)$ denotes the set of all periods of contractible Reeb orbits of $(\partial K, \alpha)$, and\vspace{0.2 cm}
        \item $h(r) = h_2(r)$ on $S^1 \times \left(\partial K \times [1+2r_1,1+3r_1]\right)$ for some strictly concave and increasing function $h_2$.\vspace{0.2 cm}
    \end{itemize}
        \item $H= c_H$ for some constant $c_H >0$ on $S^1 \times \left(M - (K \cup \left(\partial K \times [1,1+3r_1]\right)\right))$.\vspace{0.2 cm}
\end{enumerate}
 We denote the set of all contact type $K$-admissible Hamiltonian functions by $\mathcal{H}^{\text{Cont}}_K$. 

\vspace{0.2 cm}

 Since $K$ has a contact type boundary, we can choose a \textbf{cylindrical} almost complex structure $J$ around $\partial K$ on $M$, which is compatible with $\omega$ and satisfies
 \begin{align*}
     dr \circ J = -r \alpha.
 \end{align*}
With all the Floer data above, we can see that the definition \eqref{reldef} can be rewritten as
 \begin{align}\label{reldef2}
      SH^*_M(K) = H^* \left( \widehat{\varinjlim_{H \in \mathcal{H}^{\text{Cont}}_K}} CF(H) \right)
 \end{align}
where $CF(H) = CF(\widetilde{H})$ and $\widetilde{H}$ is a small nondegenerate perturbation of $H$ with $s_H = s_{\widetilde{H}}$. 

\vspace{0.2 cm}
 
Let us introduce one more terminology which is convenient for us. We say that a Hamiltonian function $H : S^1 \times M \to \RR$ is \textbf{contact type $K$-semi-admissible} if it satisfies $(A1-A3)$ but $(A1)$ is replaced by\vspace{0.2 cm}
 \begin{enumerate}[label=(A1$^{'}$)]
     \item $H = 0$ on $S^1 \times K$.\vspace{0.2 cm}
 \end{enumerate}
Note that $K$-admissible Hamiltonian function is not $K$-semi-admissible. The slope $s_H$, $y$-intercept $i_H$ and the constant $c_H$ of $K$-semi-admissible Hamiltonian function $H$ satisfy the following relations. We can show that 
\begin{align*}
    s_H \leq i_H
\end{align*}
because
\begin{align*}
    -i_H &= -(1+r_1)h_1'(1+r_1) + h_1(1+r_1)&&y\textrm{-intercept}\\
    &=-(1+r_1)s_H + \int_{1}^{1+r_1} h_1'(r) dr &&h(1)=0\\
    &\leq -(1+r_1)s_H + \int_{1}^{1+r_1} s_H\, dr&&h_1\,\,\textrm{is convex on}\,\,[1,1+r_1]\\
    &=-s_H.
\end{align*}
Also, we have an estimate of $c_H$ in terms of $s_H$:
\begin{align}\label{chsh}
    c_H \leq 3r_1 s_H.
\end{align}
Indeed, 
\begin{align*}
    c_H &= h(1+3r_1) - h(1) \\
    &=\int_{1}^{1+3r_1} h'(r) dr \\
    &=\int_{1}^{1+r_1} h_1'(r)dr+ \int_{1+r_1}^{1+2r_1}s_H\,dr+ \int_{1+2r_1}^{1+3r_1}h_2'(r)dr \\
    &\leq3r_1 s_H
\end{align*}
where the last inequality follows from the fact that $h_1'(r) \leq s_H$ and $h_2'(r) \leq s_H$.

\vspace{0.2cm}

For any $K$(-semi)-admissible Hamiltonian function $H$, there are 4 types of Hamiltonian orbits (after perturbation): 
\begin{itemize}
    \item critical points contained in $K$,\vspace{0.2 cm}
    \item nonconstant Hamiltonian orbits coming from $h_1$,\vspace{0.2 cm}
    \item nonconstant Hamiltonian orbits coming from $h_2$, and\vspace{0.2 cm}
    \item critical points lying outside $K$.\vspace{0.2 cm}
\end{itemize}
We refer to the first two types orbits as \textit{lower orbits}, and to the latter two types of orbits as \textit{upper orbits}.


 \vspace{0.2cm}
 
The definition \eqref{reldef2} can be further rephrased in terms of $K$-semi-admissible Hamiltonian functions by the following theorem.
\begin{proposition}\label{re1}
    Let $(M,\omega)$ be a closed symplectic manifold and $K \subset M$ be a compact domain with contact type boundary. For any contact type $K$-semi-admissible Hamiltonian function $H$, we have
    \begin{align*}
        SH_M(K) = H \left( \widehat{\varinjlim_{\sigma \to \infty}}CF(\sigma H)  \right)
    \end{align*}
\end{proposition}
\begin{proof}
Let $\{\sigma_i\}$ be an increasing sequence of positive numbers and $\{\epsilon_i\}$ be a decreasing sequence of positive numbers such that
\begin{align*}
    \lim_{i \to \infty} \sigma_i = \infty\,\,\textrm{and}\,\,\lim_{i \to \infty} \epsilon_i = 0.
\end{align*}
For each $i=1,2,3,\cdots$, define a function $H_i$ by $$H_i = s_iH - \epsilon_i.$$ Then each $H_i$ is a contact type $K$-admissible Hamiltonian function and the sequence $\{H_i\}_{i=1,2,\cdots}$ is a cofinal sequence of $\mathcal{H}_K^{\text{Cont}}$, that is, 
\begin{align*}
H_1\leq H_2\leq H_3\leq \cdots\,\,\textrm{and}\,\,
    \lim_{i \to \infty}H_i(t,p) = \begin{cases}
        0&\textrm{if}\,\,p \in K\\
        \infty &\textrm{if}\,\,p \notin K.
    \end{cases}
\end{align*}
Since $H_i$ and $s_iH$ differ by a constant $\epsilon_i$, we have
\begin{align*}
    CF(H_i) = CF(\sigma_iH)    
\end{align*}
as cochain complexes and therefore,
\begin{align}\label{semii}
    \widehat{\varinjlim_{i \to \infty}} CF(H_i) = \widehat{\varinjlim_{i \to \infty}} CF(\sigma_iH)
\end{align}
also holds as an equality of complexes. Then
\begin{align*}
    SH_M(K) &= H \left( \widehat{\varinjlim_{i \to \infty}} CF(H_i) \right)&&\textrm{Definition}\\& = H \left( \widehat{\varinjlim_{i \to \infty}} CF(\sigma_iH) \right)&&\textrm{by}\,\,\eqref{semii}.
\end{align*}
We can conclude the proof because $\{\sigma_i\}$ can be chosen any monotone increasing sequence with $\displaystyle\lim_{i \to \infty}\sigma_i = \infty$. \\    
\end{proof}
\vspace{0.2 cm}
\subsection{Filtered relative symplectic cohomology}
For the relative symplectic cohomology, we also have the action filtration, which can be defined in a natural manner. More precisely, for each $\tau \in \RR$,
\begin{align}\label{actrel}
    SH_M^{>-\tau}(K) = H \left( \widehat{\varinjlim_{H \in \mathcal{H}_K}} CF^{>-\tau}(H)  \right).
\end{align}
Obviously, if $K$ has contact type boundary, we may rewrite \eqref{actrel} as
\begin{align*}
    SH_M^{>-\tau}(K) = H \left( \widehat{\varinjlim_{H \in \mathcal{H}^{\text{Cont}}_K}} CF^{>-\tau}(H)  \right).
\end{align*}

\vspace{0.2 cm}

In the rest of this subsection, $(M, \omega)$ will be assumed to satisfy that $\omega|_{\pi_2(M)} = 0$ and hence the action, given by \eqref{act}, depends only on Hamiltonian orbits but not on their disk cappings. Let $K \subset M$ be a compact domain with contact type boundary. If a $K$(-semi)-admissible Hamiltonian function $H$ is given by $h(r)$ on $S^1 \times( {\partial K} \times[1, 1+3r_1])$, then the Hamiltonian vector field is given by
\begin{align*}
  X_H(p, r) = -h'(r) R_\alpha(p)  
\end{align*}
on $\partial K \times [1,1+3r_1]$ where $R_\alpha$ is the Reeb vector field of $(\partial K, \alpha)$. Hence, every $T$-periodic Reeb orbit $\gamma$ of $(\partial K, \alpha)$ corresponds to a 1-periodic Hamiltonian orbit $x = (\gamma, r_*)$ of $H$ with $h'(r_*) = T$ but it travels in the opposite direction. Associated with this correspondence are two Conley–Zehnder indices: one computed for $x$ as a Hamiltonian orbit and the other computed for $\gamma$ as a Reeb orbit. Denoting the former one by $\mu_{\textrm{CZ}}^{\textrm{Ham}}$ and the latter one by $\mu_{\textrm{CZ}}^{\textrm{Reeb}}$, they are related by 
\begin{align}\label{indexcomp}
    \mu_{\textrm{CZ}}^{\textrm{Ham}}(x) = n - \mu_{\textrm{CZ}}^{\textrm{Reeb}}(\gamma)
\end{align}
where $n = \frac{1}{2} \dim M$.
\vspace{0.2 cm}

If $K$ is a Liouville domain and a $K$(-semi)-admissible Hamiltonian function $H$ is given by $h(r)$ on $S^1 \times\left( \partial K \times [1,1+3r_1] \right) $, then the action of a nonconstant Hamiltonian orbit $x=(\gamma,r_*)$ of $H$ can be computed as
\begin{align}\label{for}
    \mathcal{A}_H(x) = - r_*h'(r_* ) + h (r_*). 
\end{align}
Note that \eqref{for} depends only or $r_*$. Using this observation, we can see that the actions of lower orbits of $K$-semi-admissible Hamiltonian functions are negative. Indeed, for a lower orbit $x=(\gamma,r_*)$ for $r_* \in (1,1+r_1)$ of a $K$-semi-admissible Hamiltonian function $H$, we have
\begin{align}\label{lo}
    h(r_*) \leq (r_*-1)h'(r_*)
\end{align}
because
\begin{align*}
    h(r_*)& = h(r_*) - h(1) &&h(1) = 0\\
    &= \int_{1}^{r_*} h'(r)d r\\
    &\leq \int_{1}^{r_*} h'(r_*)d r &&h\,\,\textrm{is convex on }\,\,(1,1+r_1) \\
    &= (r_*-1)h'(r_*). 
\end{align*}
It follows that the action of $x$ is negative:
\begin{align*}
    \mathcal{A}_H(x) &= - r_*h'(r_* ) + h (r_*) \leq - r_*h'(r_* ) + (r_*-1)h'(r_*) = -h'(r_*) <0
\end{align*}
where the last inequality follows from the fact that $h$ is increasing. We can obtain a similar estimate for the actions of upper orbits. Let $x=(\gamma,r_*)$ be an upper orbit of $H$, so that $r_* \in (1+2r_1,1+3r_1)$ and $h'(r_*)$ is the period of $\gamma$. Note that  
\begin{align*}
    c_H -h(r_*)& = h(1+3r_1) - h(r_*)&&c_H=h(1+3r_1)\\
    &=\int_{r_*}^{1+3r_1} h'(r) dr \\
    &\leq \int_{r_*}^{1+3r_1} h'(r_*) dr &&h\,\,\textrm{is concave on}\,\,(1+2r_1, 1+3r_1)\\
    & \leq (1+3r_1 - r_*)h'(r_*)
\end{align*}
Using this estimate, we obtain
\begin{align}\label{ua}
    \mathcal{A}_H(x) \geq c_H-(1+3r_1) \mathcal{A}(\gamma).
\end{align}
where $\mathcal{A}(\gamma)$ denotes the period of $\gamma$. Indeed, 
\begin{align*}
    \mathcal{A}_H(x) &= - r_*h'(r_* ) + h (r_*)\nonumber \\&\geq - r_*h'(r_* ) -(1+3r_1 - r_*)h'(r_*) +c_H\nonumber\\& = c_H-(1+3r_1)h'(r_*) \nonumber\\&=c_H-(1+3r_1) \mathcal{A}(\gamma). 
    \end{align*}

\vspace{0.2 cm}

Define a function $A_h : [1, 1+r_1] \cup[1+2r_1, 1+3r_1]  \to \RR$ by
\begin{align*}
    A_h(r) =
    \begin{cases}
        A_{h_1}(r) = -r h_1'(r) + h_1(r)&\textrm{if}\,\,r \in [1,1+r_1]\\
        A_{h_2}(r) = -r h_2'(r) + h_2(r)&\textrm{if}\,\,r \in [1+2r_1,1+3r_1].
    \end{cases}
    \end{align*}
    Then for any nonconstant orbit $x=(\gamma,r_*)$, the action \eqref{for} can be expressed as
   \begin{align*}
       \mathcal{A}_H(x) = A_h(r_*).
   \end{align*}
   Observe that 
   \begin{align}\label{mi}
       \frac{d}{dr} A_h (r)=  -h'(r) -r h''(r) +h'(r) = -r h''(r).
   \end{align}
    On $[1, 1+r_1]$, the function $A_h$ is decreasing by the convexity of $h_1$ and hence
\begin{align*}
    \max_{r \in [1, 1+r_1]} A_h(r) = A_{h_1}(1) = 0 \,\,\text{and}\,\, \min_{r \in [1, 1+r_1]}A_h(r) = A_{h_1}\left(1+r_1\right) = -i_H.
\end{align*}
Similarly, the function $A_h$ is increasing on $[1+2r_1,1+3r_1]$ by the concavity of $h_2$ and 
\begin{align*}
    \max_{r \in [1+2r_1,1+3r_1]} A_h(r) = A_{h_2}(1+3r_1) = c_H  \,\,\text{and}\,\,\min_{r \in [1+2r_1,1+3r_1]} A_h(r) = A_{h_2}\left(1+2r_1 \right) = -i_H.
\end{align*}

\vspace{0.2cm}

We introduce another variant of the action functional, which expresses the action in terms of periods of the Reeb orbit of $(\partial K, \alpha)$. Define $a_{H} : [0, 2 s_H] \to \RR$ by
\begin{align*}
    a_H (\tau) = \begin{cases}
         a_{h_1}(\tau)=\left(A_{h_1} \circ (h_1')^{-1}\right)(\tau) &\textrm{if}\,\,0\leq \tau \leq s_H\\
         a_{h_2}(\tau)=\left(A_{h_2} \circ (h_2')^{-1}\right)(\tau- s_H) &\textrm{if}\,\,s_H < \tau  \leq 2 s_H
    \end{cases}
\end{align*}
We first compute the derivative of $a_{h_1}$. By the chain rule, we have
\begin{align*}
    \frac{d}{d\tau}a_{h_1}(\tau)& = \frac{d}{d\tau}A_{h_1}\left((h_1')^{-1}(\tau)\right) \\&= \frac{d A_{h_1}}{dr}\left((h_1')^{-1}(\tau)\right) \frac{d (h_1')^{-1}}{d\tau}(\tau)
    \\&=-(h_1')^{-1}(\tau) \,h_1''\left((h_1')^{-1}(\tau)\right) \,\frac{d (h_1')^{-1}}{d\tau}(\tau) &&\textrm{by \eqref{mi}}.
    \end{align*}
    Since $h_1'\left((h_1')^{-1}(\tau)\right) = \tau$, we obtain
    \begin{align*}
        h_1''\left((h_1')^{-1}(\tau)\right) \,\frac{d (h_1')^{-1}}{d\tau}(\tau) = 1
    \end{align*}
    by the chain rule again and therefore
    \begin{align*}
        \frac{d}{d\tau}a_{h_1}(\tau) = -(h_1')^{-1}(\tau) .
    \end{align*}
    Consequently, we have the following bound for the derivative of $a_{h_1}$:
    \begin{align}\label{bound}
       -1-r_1 \leq \frac{d}{d\tau}a_{h_1}(\tau) \leq -1.
    \end{align}
    Moreover, the image of $a_{h_1}$ is $[-i_H,0]$. A completely analogous computation for $a_{h_2}$ yields 
\begin{align*}
    -1-3r_1 \leq \frac{d}{d\tau}a_{h_2}(\tau) \leq - 1-2r_1
\end{align*}
and the image of $a_{h_2}$ is $[-i_H,c_H]$

\vspace{0.2 cm}

The following proposition is the version of Proposition \ref{re1} adapted to the action filtrations.
\begin{proposition}\label{prop1}
    Let $H$ be a $K$-semi-admissible Hamiltonian. Then for $\tau\in\RR\cup\{\infty\}$
    \begin{align*} 
        SH^{> -\tau}_M(K) \cong H\left( \widehat{\varinjlim_{\sigma \to \infty }} CF^{> -\tau}(\sigma H) \right) 
    \end{align*}
\end{proposition}
\begin{proof}
The proof is analogous to that of Proposition \ref{re1}. Let $\{\sigma_i\}$ be an increasing sequence of positive numbers and $\{\epsilon_i\}$ be a decreasing sequence of positive numbers such that
\begin{align*}
    \lim_{i \to \infty} \sigma_i = \infty \,\,\textrm{and}\,\,\lim_{i \to \infty} \epsilon_i = 0.
\end{align*}
Then $H_i = \sigma_iH - \epsilon_i$ is a contact type $K$-admissible Hamiltonian function for each $i$ and $\{H_i\}_{i=1,2,\cdots}$ is a cofinal sequence of $\mathcal{H}_K^{\text{Cont}}$, and thus
\begin{align*}
    SH^{> -\tau}_M(K) = H \left( \widehat{\varinjlim_{i \to \infty}}CF^{-\tau}(H_i) \right).
\end{align*}
Since $H_i$ and $\sigma_i H$ differ by a constant $\epsilon_i$, their Floer complexes are identical. For each generator $x \in CF(H_i) = CF(\sigma_iH)$, the action functionals satisfy
\begin{align*}
    \mathcal{A}_{H_i}(x) = \mathcal{A}_{\sigma_iH}(x) - \epsilon_i.
\end{align*}
As a result, the action filtrations differ by a shift, and we obtain an equality of filtered chain complexes
\begin{align*}
CF^{>-\tau}(H_i)
= CF^{>-\tau+\epsilon_i}(\sigma_i H).
\end{align*}
If $\tau \notin \textrm{Spec}(\partial K, \alpha) $, then we may assume that $\tau-\epsilon_i \notin \textrm{Spec}(\partial K, \alpha)$ because $\displaystyle\lim_{i \to \infty} \epsilon_i = 0$. Moreover, we obtain
\begin{align*}
     CF^{>-\tau}(\sigma_iH) = CF^{>-\tau +\epsilon_i}(\sigma_iH).
\end{align*}
In the case where $\tau \in \textrm{Spec}(\partial K, \alpha) $, choose a sequence $\{\sigma_i\}$ such that $\tau$ does not coincide with the action of any 1-periodic orbit of $\sigma_iH$. Assuming $\epsilon_i>0$ is sufficiently small for each $i=1,2,3,\cdots$, we have
\begin{align*}
    CF^{>-\tau}(\sigma_iH) = CF^{>-\tau+\epsilon_i}(\sigma_iH).
\end{align*}\

\vspace{0.2cm}

In either case, we have
\begin{align*}
    CF^{>-\tau}(H_i) = CF^{>-\tau}(\sigma_iH).
\end{align*}
Therefore,
\begin{align*}
    \widehat{\varinjlim_{i \to \infty}}CF^{>-\tau}(H_i) = \widehat{\varinjlim_{i \to \infty}}  CF^{>-\tau}(\sigma_iH)
\end{align*}
as chain complexes. This completes the proof.\\
\end{proof}

To proceed further, we assume that $c_1(TM)|_{\pi_2(M)} = 0$ where $c_1(TM)$ is the first Chern class of the tangent bundle $TM$ of $M$. Then every Hamiltonian orbit has integer-valued Conley-Zehnder index so we can impose the \textit{index-boundedness} assumption on $\partial K$. For a compact domain $K \subset M$ with a contact type boundary, we say that $(\partial K, \alpha)$ is \textbf{index-bounded} if for each $k\in\ZZ$, the set of periods of Reeb orbits on $(\partial K, \alpha)$ that are contractible in $M$ and have Conley-Zehnder index $k$ of $(\partial K, \alpha)$ is bounded.

\vspace{0.2 cm}

 In the rest of this paper, we assume that $(M,\omega)$ is a closed $2n$-dimensional symplectic manifold which is symplectically aspherical, that is, $\omega|_{\pi_2(M)}=0$ and $c_1(TM)|_{\pi_2(M)}=0$, $K \subset M$ is a Liouville domain and $(\partial K, \alpha)$ is index-bounded.

\vspace{0.2cm}

Under these assumptions, the relative symplectic cohomology $SH_M(K)$ can be written in terms of lower orbits. We have seen that those lower orbits have negative actions. Moreover, if $x$ is an upper orbit corresponding to a Reeb orbit $\gamma$, then we have seen from \eqref{ua} that the action of $x$ satisfies 
\begin{align*}
    \mathcal{A}_H(x) \geq c_H - (1+3r_1) \mathcal{A}(\gamma).
\end{align*}
Since $\partial K$ is index-bounded, for each $k\in\ZZ$, there exist $b(k)>0$ such that $0< \mathcal{A}(\gamma) \leq b(k)$ for any Reeb orbit $\gamma$ on $(\partial K, \alpha)$ of Conley-Zehnder index $k$. By \eqref{indexcomp}, if a Hamiltonian orbit $x$ has Conley-Zehnder index $k$, then we have
\begin{align*}
    \mathcal{A}_H(x) \geq c_H - (1+3r_1) b(n-k).
\end{align*}

\vspace{0.2cm}

Let $\{H_i\}$ be a cofinal sequence of $\mathcal{H}_K^{\text{Cont}}$, that is, $H_1 \leq H_2 \leq H_3\leq \cdots$ and 
\begin{align*}
    \lim_{i \to \infty} H_i(t,p) = \begin{cases}
 0&\textrm{if}\,\,p \in K\\
 \infty &\textrm{if}\,\,p \in M \setminus K.
    \end{cases}
\end{align*}
Since $\displaystyle \lim_{i \to \infty} c_{H_i} = \infty$, we may assume that $c_{H_i}$ is sufficiently large so that $$c_{H_i} - (1+3r_1)b(n-k) >0$$ for fixed $k$. Denote the set of upper orbits of $H_i$ with Conley-Zehnder index $k$ by $CF^k_{U}(H_i)$.
Then $CF^k_{U}(H_i)$ is a subcomplex of $CF^k(H_i)$ because the Floer differential increases the action and every lower orbit has negative action. Define
\begin{align*}
    CF^k_{L}(H_i) = CF^k(H_i)/CF^k_{U}(H_i).
\end{align*}
Without loss of generality, we can choose $c_{H_i}>0$ so that
\begin{align*}
    c_{H_i} - (1+3r_1)b\left(n-(k-1)\right) >0\,\,\textrm{and}\,\,c_{H_i} - (1+3r_1)b(n-(k+1)) >0
\end{align*}
and hence $CF^j_{U}(H_i)$ and $CF^j_{L}(H_i)$ make sense for $j = k-1,k,k+1$.

\vspace{0.2 cm}

For any $\ZZ$-graded cochain complex $(C,d)$ and $k\in \ZZ$, define a cochain complex $(C^{\sim k},d)$ by
\begin{align*}
    \left(C^{\sim k}\right)^i = \begin{cases}
        C^k&\textrm{for}\,\,i=k-1,k,k+1\\
        0&\textrm{otherwise}.
    \end{cases}
\end{align*}
Note that $k$-th cohomology of $(C^{\sim k},d)$ is $H^k(C,d)$, that is,
\begin{align}\label{kth}
 H^k(C,d) = H^k(C^{\sim k},d).   
\end{align}
 
\begin{proposition}\label{lower}
 For each $k\in\ZZ$, let $\{H_i\}$ be a cofinal sequence of $\mathcal{H}_K^{\text{Cont}}$ described above. Then, for any $\tau \in \RR$, we have an isomorphism
    \begin{align*}
        SH^{>-\tau,k}_M(K) \cong H^k \left( \varinjlim_{i \to \infty} CF^{>-\tau}_{L}(H_i) \right).
    \end{align*}
        
\end{proposition}

\begin{proof}

\vspace{0.2 cm}

    Let us consider the following short exact sequence:
    \begin{align*}
        0 \lr CF^{>-\tau,\sim k}_{U}(H_i) &\lr CF^{>-\tau,\sim k}(H_i)\\& \lr CF^{>-\tau,\sim k}(H_i)/CF^{>-\tau,\sim k}_{U}(H_i) = CF^{>-\tau,\sim k}_{L}(H_i) \lr 0.
    \end{align*}
Since the direct limit is an exact functor, we still have the exact sequence as follows:
\begin{align*}
    0 \lr \varinjlim_{i \to \infty}CF^{>-\tau,\sim k}_{U}(H_i) \lr \varinjlim_{i \to \infty}CF^{>-\tau,\sim k}(H_i) \lr \varinjlim_{i \to \infty}CF^{>-\tau,\sim k}_{L}(H_i) \lr 0.
\end{align*}
Note that each $CF^{>-\tau,\sim k}_{L}(H_i)$ is a free module over $\Lambda_{\geq 0}$ generated by lower orbits of $H_i$ with action greater than $-\tau$ of Conley-Zehnder index $k-1,k$ and $k+1$ and hence it is a flat module over $\Lambda_{\geq0}$. Since the direct limit of flat modules is known to be flat, we have
\begin{align*}
    \text{Tor}_1^{\Lambda_{\geq 0}}\left(\varinjlim_{i \to \infty}CF^{>-\tau,\sim k}_{L}(H_i), \Lambda_{\geq 0}/ \Lambda_{\geq r}\right) = 0
\end{align*}
for each $r \geq 0$ and therefore,
\begin{align*}
    0 \lr  \varinjlim_{i \to \infty}CF^{>-\tau,\sim k}_{U}(H_i) \otimes \Lambda_{\geq 0}/\Lambda_{\geq r} & \lr \varinjlim_{i \to \infty}CF^{>-\tau,\sim k}(H_i) \otimes \Lambda_{\geq 0}/\Lambda_{\geq r}\\ & \lr \varinjlim_{i \to \infty}CF^{>-\tau,\sim k}_{L}(H_i)\otimes \Lambda_{\geq 0}/\Lambda_{\geq r} \lr 0
\end{align*}
is still an exact sequence. For $0 \leq r \leq r'$, the projection
\begin{align}\label{mlex}
    \Lambda_{\geq 0}/ \Lambda_{\geq r'} \to \Lambda_{\geq 0}/ \Lambda_{\geq r}
\end{align}
is surjective. For any $\Lambda_{\geq 0}$-module $A$, the map induced by \eqref{mlex} $$A \ton \Lambda_{\geq 0}/ \Lambda_{\geq r'} \to A \ton \Lambda_{\geq 0}/ \Lambda_{\geq r}$$ induced by \eqref{mlex} is also surjective due to the right exactness of tensor product. Then by the Mittag-Leffler theorem for the inverse limit, the following sequence
\begin{align*}
    0 \lr  \varprojlim_{r \to 0} \varinjlim_{i \to \infty}CF^{>-\tau,\sim k}_{U}(H_i) \otimes \Lambda_{\geq 0}/\Lambda_{\geq r} & \lr \varprojlim_{r \to 0}\varinjlim_{i \to \infty}CF^{>-\tau,\sim k}(H_i) \otimes \Lambda_{\geq 0}/\Lambda_{\geq r}\\ & \lr \varprojlim_{r \to 0}\varinjlim_{i \to \infty}CF^{>-\tau,\sim k}_{L}(H_i)\otimes \Lambda_{\geq 0}/\Lambda_{\geq r} \lr 0
\end{align*}
is also exact. By the definition of the completion (see Section \ref{floer}), we have the following exact sequence of complexes
\begin{align}\label{lmc}
    0 \lr  \widehat{\varinjlim_{i \to \infty}}CF^{>-\tau,\sim k}_{U}(H_i) \lr \widehat{\varinjlim_{i \to \infty}}CF^{>-\tau,\sim k}(H_i)\lr \widehat{\varinjlim_{i \to \infty}}CF^{>-\tau,\sim k}_{L}(H_i) \lr 0.
\end{align}
    Since it is proved in \cite{dgpz} that\vspace{0.2cm}
    \begin{itemize}
        \item (Lemma 5.3 of \cite{dgpz})\, $ \displaystyle\widehat{\varinjlim_{i \to \infty}}CF^{>-\tau,\sim k}_{U}(H_i) = 0$, and\vspace{0.2cm}
        
        \item (Lemma 5.4 of \cite{dgpz}) \, $\displaystyle \widehat{\varinjlim_{i \to \infty}}CF^{>-\tau,\sim k}_{L}(H_i) \cong \varinjlim_{i \to \infty}CF^{>-\tau,\sim k}_{L}(H_i)$,\\
    \end{itemize}
  the short exact sequence \eqref{lmc} reduces to the isomorphism of complexes
  \begin{align}\label{reduce}
      \widehat{\varinjlim_{i \to \infty}}CF^{>-\tau,\sim k}_{}(H_i) \cong \varinjlim_{i \to \infty}CF^{>-\tau,\sim k}_{L}(H_i).
  \end{align}
  Then
  \begin{align*}
      SH^{>-\tau,k}_M(K) &= H^k \left( \widehat{\varinjlim_{i \to \infty} }CF^{>-\tau}(H_i)  \right) &&\textrm{Definition}\\
      &= H^k \left( \widehat{\varinjlim_{i \to \infty} }CF^{>-\tau, \sim k}(H_i) \right)&&\textrm{by}\,\,\eqref{kth}\\
      &\cong H^k \left( \varinjlim_{i \to \infty} CF^{>-\tau, \sim k}_{L}(H_i) \right)&&\textrm{by}\,\,\eqref{reduce}\\
      &= H^k \left( \varinjlim_{i \to \infty} CF^{>-\tau}_{L}(H_i) \right) &&\textrm{by}\,\,\eqref{kth}.
  \end{align*}
   This completes the proof.\\ 
\end{proof}
\begin{corollary}\label{cor3}
    Let $H$ be a $K$-semi-admissible Hamiltonian function and $\tau \in \RR $. Then for each $k \in \ZZ$, we have
    \begin{align*}
        SH^{>- \tau,k}_M(K) &\cong H^k \left(  \varinjlim_{\sigma \to \infty} CF^{>- \tau}_{L}(\sigma H) \right)\\
        &\cong H^k \left( \varinjlim_{\sigma \to \infty} CF^{(-\tau,0]}(\sigma H)  \right)\\& \cong \varinjlim_{\sigma \to \infty} HF^{(-\tau,0],k}(\sigma H).
    \end{align*}
\end{corollary}
\begin{proof}
In view of the proof of Proposition \ref{lower}, it suffices to check that every upper orbit of $\sigma H$ of Conley-Zehnder index $k-1,k$ and $k+1$ has positive action. Let $x$ be such an orbit. Then by \eqref{ua}, we have
\begin{align*}
    \mathcal{A}_{\sigma H}(x) &\geq c_{\sigma H} - (1+3r_1) \max\left\{b(n-k+1),b(n-k),b(n-k-1)\right\}.
\end{align*}
Since $c_{\sigma H} = \sigma c_H$, this yields
\begin{align*}
    \mathcal{A}_{\sigma H}(x) \geq \sigma c_{H} - (1+3r_1) \max\left\{b(n-k+1),b(n-k),b(n-k-1)\right\}>0
\end{align*}
for sufficiently large $\sigma >0$. The remainder of the argument follows verbatim from the proof of Proposition \ref{lower}, and we therefore omit the details. The second isomorphism follows from the fact that all lower orbits have negative action, while the final isomorphism uses the commutativity of taking direct limits with homology.\\
\end{proof}
Proposition \ref{lower} and Corollary \ref{cor3} allow us to set
    \begin{align*}
        SH^{> -\tau}_M(K)\,\,\textrm{if}\,\,\tau \leq 0.
    \end{align*}
    Also, if $x=(\gamma,r_*)$ is a lower orbit of $H$ and has action greater than $-\tau$, then 
    \begin{align*}
        -\tau &< -r_* h'(r_*) +h(r_*)\\
        &\leq -r_* h'(r_*) +(r_* -1) h'(r_*)&&\textrm{by \eqref{lo}}\\
        &=- \mathcal{A}(\gamma)
    \end{align*}
    and therefore $SH^{> -\tau}_M(K)$ is generated by Reeb orbits of $(\partial K, \alpha)$ with period less than $\tau$.

\vspace{0.2cm}
    
Furthermore, in situations where the action window can be effectively controlled, the passage to the direct limit becomes unnecessary, as the corresponding Floer complexes already capture the desired information.
\begin{theorem}\label{hamiso}
For each integer $k \in \ZZ$, there exists a contact type $K$-semi-admissible Hamiltonian function $H$ such that 
    \begin{align*}
        SH^{>-\tau,k}_M(K) \cong HF^{(a_{H}(\tau),0],k}(H)
    \end{align*}
    for $0 \leq \tau \leq s_H$.
\end{theorem}
\begin{proof}
Let $H$ be a contact type $K$-semi-admissible Hamiltonian function satisfying the following condition: Choose a constant $c_{H} >0$ sufficiently large so that the upper orbits of $H$ of Conley-Zehnder index $k-1,k$ and $k+1$ have positive actions. As shown above, this choice is possible thanks to the index-boundedness of $(\partial K, \alpha)$ together with the action estimate \eqref{ua} for upper orbits. For $\sigma>1$, we shall prove that 
\begin{align}
    HF^{(a_H(\tau),0],k}(H) \cong HF^{(a_{\sigma H}(\tau),0],k}(\sigma H)
\end{align}
for $0\leq \tau \leq s_H$. Fix $\tau \in [0,s_H]$. Then there exists a unique $r_0\in [1,1+r_1]$ such that $A_H(r_0) = a_H(\tau)$, that is, $h'(r_0)=\tau$. Consider a $K$-semi-admissible Hamiltonian function $G$ such that\vspace{0.1 cm}
\begin{itemize}
    \item $H \leq G \leq \sigma H$,\vspace{0.2 cm}
    \item $G = H$ on $[1, r_0]$, and\vspace{0.2 cm}
    \item $s_{G} = s_{\sigma H} = \sigma s_H$.\vspace{0.1 cm}
\end{itemize}
Note that $a_H(\tau) = a_G(\tau)$ and $g'(r_0) =\tau$. The monotone homotopy from $H$ to $G$ induces an isomorphism
\begin{align}\label{link1}
    HF^{(a_H(\tau),0],k}(H) \cong HF^{(a_H(\tau),0],k}(G) 
\end{align}
Indeed, the generators of the two complexes coincide, and, by the maximum principle, every Floer trajectory connecting 1-periodic orbits contained in $K_{1+r_1} = K \cup \left( \partial K \times [1,1+r_1] \right)$ has its image entirely in $K_{1+r_1}$. 

\vspace{0.2 cm}

Let $$F_{\sigma,t} = (1-t)G + t (\sigma H)\,\, ,\,\,0\leq t \leq 1$$ be the increasing linear homotopy from $G$ to $\sigma H$. Observe that the slope of $F_{\sigma,t}$ is independent of $t$; in fact, $s_{F_{\sigma,t}} = \sigma s_H$

\vspace{0.2 cm}

We first consider the case when $a_{H}(\tau) \notin \text{Spec}(H)$. Since $H =G$ on $K \cup (\partial K \times [1,r_0])$ and $h'(r_0) = \tau = g'(r_0)$, we have $a_H(\tau) = a_G(\tau) \notin \textrm{Spec}(G)$. Consider a function
\begin{align*}
    \phi_{\sigma,t} = a_{F_{\sigma,t}} \circ a_G^{-1} : [-i_G, 0] \lr [0, s_G]= [0, s_{F_{\sigma,t}}] \lr [-i_{F_{\sigma,t}},0].
\end{align*}
According to \cite{cggm2}, this map $\phi_{\sigma,t}$ sends the spectra of $G$ to the spectra of $F_{s,t}$. More precisely, $$\phi_{\sigma,t}(\textrm{Spec}(G)) = [-i_{F_{\sigma,t}},0] \cap \textrm{Spec}(F_{\sigma,t}).$$ Thus we obtain
\begin{align}\label{isoing}
    a_{F_{\sigma,t}}(\tau) = \phi_{\sigma,t}(a_H(\tau)) \notin \textrm{Spec}(F_{\sigma,t})
\end{align}
for all $t\in[0,1]$. Since $a_{F_{\sigma,t}}(\tau)$ is continuous with respect to $t$, for each $t \in [0,1]$, we can choose $\delta(t) > 0$ such that 
\begin{align*}
    HF^{(a_{F_{\sigma,t}}(\tau),0],k}(F_{\sigma,t}) \cong HF^{(a_{F_{\sigma,t'}}(\tau),0],k}(F_{\sigma,t})\,\,\textrm{for}\,\,|t-t'| < \delta(t).
\end{align*}
We can choose finitely many points $0=t_0 < t_1 < t_2 < \cdots< t_N = 1$ such that \vspace{0.2 cm}
\begin{itemize}
    \item $HF^{(a_{F_{\sigma,t_i}}(\tau),0],k}(F_{\sigma,t_i}) \cong HF^{(a_{F_{\sigma,t'}}(\tau),0],k}(F_{\sigma,t_i})$ for $|t_i-t'| < \delta(t_i) :=\delta_i$, and\vspace{0.3 cm}
    \item $t_{i+1} - \delta_{i+1} < t_{i} + \delta_{i}$ and hence there exists $\eta_i \in (t_{i+1} - \delta_{i+1}, t_{i} + \delta_{i})$ for each $i=0,1,2,\cdots,N-1$.\vspace{0.2 cm}
\end{itemize}
For notational convenience, let $I_{t} = (a_{F_{\sigma,t}}(\tau),0]$. Then by the choice of $\eta_i$, we have
\begin{align}\label{ibs}
    HF^{I_{t_i},k}(F_{\sigma,t_i}) \cong HF^{I_{\eta_i},k}(F_{\sigma,t_i}).
\end{align}
Therefore, for each $i = 0,1,2, \cdots, N-1$, we obtain the following isomorphism
\begin{align}\label{mola}
    HF^{I_{t_i},k}(F_{\sigma,t_i}) \cong HF^{I_{t_{i+1}},k} (F_{\sigma,t_{i+1}}).
\end{align}
Indeed, for each $i$,
\begin{align*}
    HF^{I_{t_i},k}(F_{\sigma,t_i}) &\cong HF^{I_{\eta_i},k} (F_{\sigma,t_i}) &&\textrm{by}\,\,\eqref{ibs}\\
    &\cong HF^{I_{\eta_i},k} (F_{\sigma,t_{i+1}})&&s_{F_{\sigma,t_i}} = s_{F_{\sigma,t_{i+1}}}\\
    &\cong HF^{I_{t_{i+1}},k} (F_{\sigma,t_{i+1}}) &&\textrm{by}\,\,\eqref{ibs}.
\end{align*}
Consequently, 
\begin{align}\label{link2}
    HF^{(a_H(\tau),0],k}(G) = HF^{I_{t_0},k}(F_{s,t_0})\cong \cdots \cong HF^{I_{t_N},k}(F_{s,t_N}) = HF^{(a_{\sigma H}(\tau),0],k}(\sigma H)
\end{align}
where the intermediate isomorphisms follow from \eqref{mola}. Combining \eqref{link1} and \eqref{link2} together, we obtain
\begin{align*}
     HF^{(a_H(\tau),0],k}(H) \cong HF^{(a_{\sigma H}(\tau),0],k}(\sigma H). 
\end{align*}

\vspace{0.2 cm}

Now assume that $a_H(\tau) \in \textrm{Spec}(H)$. We can choose $\epsilon_0 >0$ such that $a_H(\tau+\epsilon ) \notin \textrm{Spec}(H)$ for $0< \epsilon \leq \epsilon_0$. Then by the argument above, we have an isomorphism
\begin{align*}
    HF^{(a_H(\tau +\epsilon), 0],k}(H) \cong HF^{(a_{\sigma H}(\tau +\epsilon),0],k}(\sigma H).
\end{align*}
For $0< \epsilon \leq \epsilon' \leq \epsilon_0$, there exists a map 
\begin{align*}
    HF^{(a_H(\tau+\epsilon ),0],k}(H) \lr HF^{(a_H(\tau+\epsilon' ),0],k}(H)
\end{align*}
induced by the inclusion. It follows that $\{ HF^{(a_H(\tau+\epsilon ), 0],k}(H)  \}_{\epsilon \in (0, \epsilon_0]}$ forms a direct system and 
\begin{align*}
    \varinjlim_{\epsilon \to 0} HF^{(a_H(\tau+\epsilon ), 0],k}(H) = HF^{(a_H(\tau),0],k}(H). 
\end{align*}
The same reasoning still holds for $HF^{(a_{\sigma H}(\tau +\epsilon), 0],k}(\sigma H)$ and therefore we have
\begin{align*}
    HF^{(a_H(\tau),0],k}(H) \cong HF^{(a_{\sigma H}(\tau),0],k}(\sigma H).
\end{align*} 

\vspace{0.2cm}

 To recapitulate, in any cases, we have the following isomorphism:
\begin{align}\label{limi}
    HF^{(a_H(\tau),0],k}(H) \cong HF^{(a_{\sigma H}(\tau),0],k}(\sigma H).
\end{align}
Moreover, the isomorphism \eqref{limi} is constructed as a composition of intermediate maps, each of which is either an inclusion map or a continuation map.

\vspace{0.2 cm}

For $\sigma_1 \leq \sigma_2$, we have a map
\begin{align*}
    HF^{(a_{\sigma_1H}(\tau),0],k}(\sigma_1H) \lr HF^{(a_{\sigma_2H}(\tau),0],k}(\sigma_1H) \lr HF^{(a_{\sigma_2H}(\tau),0],k}(\sigma_2H) 
\end{align*}
where the first map is induced by inclusion and the second map is the continuation map. Note that these maps are compatible with the isomorphism \eqref{limi} and that $\{ HF^{(a_{\sigma H}(\tau),0],k}(\sigma H)\}_{\sigma>1}$ forms a direct system.  As shown in \cite{cggm2}, we have
\begin{align*}
    a_{\sigma H}(\tau) \to -\tau \,\,\textrm{as}\,\,\sigma \to \infty
\end{align*}
uniformly on compact sets in $[0, \infty)$. Using this limit together with Corollary \ref{cor3}, we obtain
\begin{align*}
    \varinjlim_{\sigma \to \infty}HF^{(a_{\sigma H}(\tau),0],k}(\sigma H) = SH^{>-\tau,k}_M(K). 
\end{align*}
Finally, taking $\displaystyle\varinjlim_{\sigma \to \infty}$ on both sides of \eqref{limi} completes the proof.\\
\end{proof}
\begin{corollary}\label{hamisocoro}
    Let $H$ be a $K$-semi-admissible Hamiltonian function. Then for each $k \in \ZZ$, there exists a constant $\sigma >0$ such that
    \begin{align*}
        SH^{>-\tau,k}_M(K) \cong HF^{(a_{\sigma H}(\tau),0],k}(\sigma H)
    \end{align*}
    for $0 \leq \tau \leq s_{\sigma H} = \sigma s_H$.
\end{corollary}
\begin{proof}
We can choose $\sigma >0$ so that all upper orbits of $\sigma H$ with Conley-Zehnder index $k-1,k$ and $k+1$ have positive actions as in the arguments used earlier. With this choice, the argument proceeds analogously to the proof of Theorem~\ref{hamiso}, and we therefore omit the details.\\    
\end{proof}

\begin{remark}\label{rmk1}
The proofs of Theorem \ref{hamiso} and Corollary \ref{hamisocoro} also imply the following statements.
\begin{enumerate}[label=(\alph*)]
    \item Let $A \subset \ZZ$ be a finite subset of $\ZZ$. \vspace{0.2cm}
    \begin{itemize}
       \item There exists a contact type $K$-semi-admissible Hamiltonian function $H$ such that there exists an isomorphism
    \begin{align*}
        SH^{>-\tau,k}_M(K) \cong HF^{(a_{H}(\tau),0],k}(H)
    \end{align*}
    for $0 \leq \tau \leq s_H$ and for $k \in A$. \vspace{0.2cm}
    \item For any contact type $K$-semi-admissible Hamiltonian function $H$, there exists constant $\sigma >0$ such that
    \begin{align*}
         SH^{>-\tau,k}_M(K) \cong HF^{(a_{\sigma H}(\tau),0],k}(\sigma H)
    \end{align*}
    for $0 \leq \tau \leq s_{\sigma H}$ and for $k \in A$.\vspace{0.2cm}
    \end{itemize}
    \item Let $H$ be a contact type $K$-semi-admissible Hamiltonian function and let
    \begin{align*}
        I_H = \left\{ k \in \ZZ \bigmid \substack{ \textstyle \text{the actions of upper orbits of } H \text{ of Conley-Zehnder indices } $\vspace{0.2cm}$\\ 
             \,\,\textstyle k-1, k \text{ and } k+1 \text{ are positive}} \right\}.
    \end{align*}
    Then there exists an isomorphism 
    \begin{align*}
        SH_M^{>-\tau,k}(K) \cong HF^{(a_{H}(\tau),0],k}(H).
    \end{align*}
    for $0\leq\tau\leq s_H$ and for $k \in I_H$.
   \end{enumerate}
       \end{remark}

\vspace{0.2cm}  

\section{Barcode entropies}
\subsection{Definition}

For a persistence module $V$, there exists a unique barcode $\mathcal{B}(V)$ of $V$. See the normal form theorem (Theorem \ref{nft}). For any $\epsilon>0$, let
\begin{align*}
    \mathcal{B}_{\epsilon}(V) = \{\textrm{bars in}\,\,\mathcal{B}(V)\,\,\textrm{of length greater than}\,\,\epsilon  \}
\end{align*}
and $b_{\epsilon}(V)$ be the number of bars in $\mathcal{B}_{\epsilon}(V)$, i.e.,
\begin{align*}
   b_{\epsilon}(V) = | \mathcal{B}_{\epsilon}(V)|.
\end{align*}
\begin{definition}\label{entro}
    Let $V$ be a persistence module. For any $\epsilon >0$, the \textbf{$\epsilon$-barcode entropy} $\hbar_{\epsilon}(V)$ of $V$ is defined to be
\begin{align*}
    \hbar_{\epsilon}(V) = \limsup_{\sigma \to \infty} \frac{1}{\sigma}\log^+ b_{\epsilon}(\textrm{tru}(V,\sigma)),
\end{align*}
where $\log^+a = \max\{\log_2 a, 0\}$. The \textbf{barcode entropy} $\hbar(V)$ of $V$ is defined by
    \begin{align*}
        \hbar(V) = \lim_{\epsilon \to 0} \hbar_{\epsilon}(V).
    \end{align*}
    \end{definition}

\vspace{0.2 cm}

It follows from Theorem \ref{hamiso} that the assignment
$$\tau \mapsto SH^{>-\tau,k}_M(K)$$ defines a persistence module, whose structure maps
$$SH^{>-\tau_1,k}_M(K) \lr SH^{>-\tau_2,k}_M(K)$$ for $\tau_1 \leq\tau_2$ are induced by the inclusion maps. We denote this persistence module by $SH^k_M(K)$. Let $\mathcal{B}(SH^k_M(K))$ be the barcode of the persistence module $SH^k_M(K)$. Since 
\begin{align*}
    SH^{>-\tau}_M(K) = \bigoplus_{k \in \ZZ} SH^{>-\tau,k}_M(K),
\end{align*}
the barcode $\mathcal{B}(SH_M(K))$ of $SH_M(K)$ is defined to be
\begin{align*}
    \mathcal{B}(SH_M(K)) = \coprod_{k \in \ZZ} \mathcal{B}(SH^k_M(K)).
\end{align*}

\vspace{0.2 cm}

Following Definition \ref{entro}, we define the \textbf{relative symplectic cohomology barcode entropy} $\hbar(SH_M(K))$ of relative symplectic cohomology $SH_M(K)$ by
\begin{align*}
    \hbar(SH_M(K)) = \lim_{\epsilon \to 0} \hbar_{\epsilon} (SH_M(K))
\end{align*}
where
\begin{align*}
    \hbar_{\epsilon} (SH_M(K)) = \limsup_{\sigma \to \infty} \frac{1}{\sigma}\log^+b_{\epsilon}(\textrm{tru}(SH_M(K), \sigma)).
\end{align*}
\vspace{0.2cm}

We can also define the barcode entropy associated with a Hamiltonian function. For a contact type $K$-semi-admissible Hamiltonian function $H$, the assignment $$\tau \mapsto HF^{>-\tau}(H)$$ defines a persistence module and we denote this persistence module by $HF(H)$. We define the \textbf{Hamiltonian barcode entropy} $\hbar(HF(H))$ of $H$ by
\begin{align*}
    \hbar(HF(H)) = \lim_{\epsilon \to 0} \hbar_\epsilon(HF(H))
\end{align*}
where
\begin{align*}
    \hbar_\epsilon(HF(H)) = \limsup_{\sigma \to \infty} \frac{1}{\sigma} \log^+ b_\epsilon \left( \textrm{tru} (HF(\sigma H), \sigma i_H) \right).
\end{align*}
In view of Theorem \ref{hamiso}, we consider the persistence module $\tau \mapsto HF^{(-\tau,0]}(H)$, which we denote by $HF^{\textrm{neg}}(H)$. The \textbf{negative Hamiltonian barcode entropy} $\hbar(HF^{\textrm{neg}}(H))$ of $H$ is defined by
\begin{align*}
    \hbar(HF^{\textrm{neg}}(H)) = \lim_{\epsilon \to 0}\hbar_\epsilon(HF^{\textrm{neg}}(H))
\end{align*}
where
\begin{align*}
    \hbar_\epsilon(HF^{\textrm{neg}}(H)) = \limsup_{\sigma \to \infty} \frac{1}{\sigma} \log^+ b_\epsilon \left( \textrm{tru} (HF^{\textrm{neg}}(\sigma H), \sigma i_H) \right).
\end{align*}

\vspace{0.2 cm}
\subsection{Comparison of barcode entropies}
We would like to compare the barcode entropies defined above in this subsection. The following theorem elucidates the interplay between the relative symplectic cohomology barcode entropy and the Hamiltonian barcode entropy.
 
\begin{theorem}\label{rene}
    There exists a contact type $K$-semi-admissible Hamiltonian function $H$ such that for any $\epsilon>0$, 
         \begin{align*}
       \hbar_{\epsilon}(SH_M(K)) \leq \hbar_{\epsilon}(HF^{\textrm{neg}}(H))
    \end{align*}
    and hence, taking $\epsilon \to 0$, we have  
    \begin{align*}
       \hbar(SH_M(K)) \leq \hbar(HF^{\textrm{neg}}(H)).
        \end{align*}

\end{theorem}
\begin{proof}
Choose a $K$-semi-admissible Hamiltonian function $H$ satisfying
\begin{align}\label{ch}
    c_H \geq 1+3r_1
\end{align}
Observe that the condition \eqref{ch} only depends on $r_1$, that is, on the position of $K$ inside $M$ and it forces to $s_H$ to be greater than 1. Indeed, if $s_H \leq 1$, then
\begin{align*}
    1+3r_1 &\leq c_H &&\textrm{assumption}\\
    &\leq 3r_1 s_H &&\textrm{by \eqref{chsh}}\\
    &\leq 3r_1,
\end{align*}
which yields a contradiction.

\vspace{0.2cm}

For each $\sigma>0$, the number of bars $b_{\epsilon} \left( \textrm{tru}(SH_M(K),\sigma ) \right)$ is finite and hence only finitely many grading degrees contribute nontrivially to it. Consequently, there exists a finite subset $A_\sigma$ of $\ZZ$ such that
\begin{align*}
    b_{\epsilon} \left( \textrm{tru}(SH_M(K),\sigma ) \right) = b_{\epsilon} \left( \textrm{tru}\left(\bigoplus_{k \in A_\sigma}SH^k_M(K),\sigma \right) \right) = \sum_{k \in A_\sigma}b_{\epsilon} \left( \textrm{tru}(SH^k_M(K),\sigma ) \right).
\end{align*}
To estimate $b_{\epsilon} \left( \textrm{tru}(SH_M(K),\sigma ) \right)$, it suffices to consider Reeb orbits of $(\partial K,\alpha)$ with periods less than $\sigma - \epsilon$. We claim that
\begin{align}\label{ingiso}
     SH^{>-\tau,k}_M(K) \cong  HF^{(a_{\sigma H}(\tau),0],k}(\sigma H)
\end{align}
for $0\leq \tau\leq\sigma \leq \sigma s_H$ and $k \in A_\sigma$. To justify this claim, it is enough to show that the upper orbits of $\sigma H$ have positive actions in view of the proof of Theorem \ref{hamiso}, Corollary \ref{hamisocoro} and Remark \ref{rmk1}. Let $x$ be an upper orbit of $\sigma H$ corresponding to a Reeb orbit $\gamma$ with $\mathcal{A}(\gamma) < \sigma - \epsilon$. Then we obtain
\begin{align*}
   \mathcal{A}_{\sigma H}(x) &\geq c_{\sigma H} - (1+3r_1)\mathcal{A}(\gamma) &&\textrm{by \eqref{ua}}\\&> \sigma c_H- (1+3r_1)(\sigma -\epsilon)\\&>0 &&\textrm{by the choice of }\,\,H,
\end{align*}
and this proves the claim.

\vspace{0.2cm}

Recall from \eqref{bound} that for $0 \leq \tau_1 \leq \tau_2 \leq \sigma \leq \sigma s_H = s_{\sigma H}$,
\begin{align*}
\tau_2 - \tau_1 \leq a_{\sigma H}(\tau_1) - a_{\sigma H}(\tau_2) \leq (1+r_1)(\tau_2 - \tau_1),
\end{align*}
and this together with the isomorphism \eqref{ingiso} implies that 
\begin{align}\label{oy}
    b_{\epsilon} \left( \textrm{tru}\left( \bigoplus_{k \in A_\sigma} SH^{k}_M(K),\sigma \right) \right) \leq b_{\epsilon} \left( \textrm{tru}\left(\bigoplus_{k \in A_\sigma}HF^{\textrm{neg},k}(\sigma H),-a_{\sigma H}(\sigma)\right) \right).
\end{align}
Therefore,
\begin{align*}
      \hbar_\epsilon(SH_M(K)) &= \limsup_{\sigma \to \infty}\frac{1}{\sigma} \log^+b_{\epsilon} \left( \textrm{tru}(SH_M(K),\sigma) \right)\\ 
      &=\limsup_{\sigma \to \infty}\frac{1}{\sigma} \log^+b_{\epsilon} \left( \textrm{tru}\left(\bigoplus_{k \in A_\sigma}SH^k_M(K),\sigma)\right) \right)
      \\& \leq\limsup_{\sigma \to \infty}\frac{1}{\sigma} \log^+b_{\epsilon}\left( \textrm{tru}\left(\bigoplus_{k \in A_\sigma}HF^{\textrm{neg},k}(\sigma H),-a_{\sigma H}(\sigma)\right) \right)&&\textrm{by}\,\,\eqref{oy}\\
      & \leq\limsup_{\sigma \to \infty}\frac{1}{\sigma} \log^+b_{\epsilon}\left( \textrm{tru}\left(\bigoplus_{k \in A_\sigma}HF^{\textrm{neg},k}(\sigma H),\sigma i_H\right) \right)&&-a_{\sigma H}(\sigma) \leq  i_{\sigma H}\\
      &\leq\limsup_{\sigma \to \infty}\frac{1}{\sigma} \log^+b_{\epsilon}\left( \textrm{tru}\left(HF^{\textrm{neg}}(\sigma H),\sigma i_H\right) \right)\\
      &= \hbar_{\epsilon}(HF^{\textrm{neg}}(H)).
\end{align*}
This completes the proof.

\end{proof}

\section{Upper bound of the relative symplectic cohomology barcode entropy}
\subsection{Lagrangian tomograph and Crofton's inequality}
We briefly discuss the idea of \textit{Lagrangian tomograph}. Let $(M',\omega')$ be a symplectic manifold and $L \subset M'$ be a Lagrangian submanifold. For a compact manifold $B$ possibly with boundary, a smooth map $$\Psi : B \times L \lr M$$ is called a \textbf{Lagrangian tomograph} of $L$ if \vspace{0.1 cm}
\begin{itemize}
    \item it is a submersion,\vspace{0.2 cm}
    \item the map $\Psi_s : L \lr M$ given by $  \Psi_s(x) = \Psi(s,x)$ is an embedding for each $s \in B$, \vspace{0.2 cm}
    \item the image $L_s = \Psi_s(L)$ is a Lagrangian submanifold of $(M',\omega')$ for each $s \in B$, and\vspace{0.2 cm}
    \item every $L_s$ is Hamiltonian isotopic to each other.\vspace{0.1 cm}
\end{itemize}
We call the Lagrangian submanifold $L$ the \textbf{core} of the tomograph $\Psi$. The existence of such a Lagrangian tomograph can be found in Lemma 5.6 of \cite{cgg}; A Lagrangian tomograph with its core $L$ and $\dim B = d$ exists if and only if $L$ admits an immersion into $\RR^d$. Moreover, they proved the following version of \textit{Crofton's inequality}.
\begin{theorem}[Crofton's inequality]\label{crof}
    Let $(M',\omega')$ be a symplectic manifold and $L_1, L_2 \subset M'$ be Lagrangian submanifolds with $\dim L_1=\operatorname{codim} L_2 $. Let $\Psi : B \times L_1 \lr M$ be a Lagrangian tomograph of $L_1$ and $ds$ be a smooth measure of $B$. Then there exists a constant $D >0$ depending on $\Psi$, $ds$ and a fixed metric on $M'$ but not depending on $L_2$ such that
    \begin{align*}
        \int_B N(s) ds \leq D\, \textrm{Vol}(L_2)
    \end{align*}
    where $N(s) = | L_s \cap L_2|$ and Vol denotes the volume.
\end{theorem}
\begin{proof}
    See Lemma 5.3 of \cite{cgg}.\\
\end{proof}
The integration in Theorem \ref{crof} makes sense because $\Psi_s$ is transverse to $L'$ for almost every $s \in B$ and hence $N(s) = | L_s \cap L'|$ is locally constant function outside of a measure zero set of $B$.
\subsection{Proof of Theorem \ref{thma}}
In our setting, we consider the symplectic manifold $(M', \omega') = ( M \times M,(-\omega)\oplus \omega)$, the Lagrangian submanifold $L = \Delta=\{(x,x) \in M \times M \mid x\in M\}$ of $M'$ and the Lagrangian tomograph $\Delta_s$ of the diagonal $\Delta$. For clarity and ease of reference, we restate Theorem \ref{thma} and proceed with its proof.
\begin{theorem}\label{thmare}
    There exists a constant $C=C(M,K) >0$, depending on the pair $(M,K)$, such that
    \begin{align*}
     \hbar (SH_M(K))\leq C\, h_{\textrm{top}}(\varphi_\alpha).
    \end{align*}
 \end{theorem}
\begin{proof}
We use the same notation as in the proof of Theorem \ref{rene}. Let $\epsilon>0$. For sufficiently large $\sigma>0$, as shown in the proof of Theorem \ref{rene}, we can choose a $K$-semi-admissible Hamiltonian function $H$ such that the isomorphism 
\begin{align*}
    \bigoplus_{k \in A_\sigma} SH^{>-\tau,k}_M(K) \cong \bigoplus_{k \in A_\sigma} HF^{(a_{\sigma H}(\tau),0],k}(\sigma H)
\end{align*}
holds for $0\leq \tau\leq\sigma \leq \sigma s_H$. Consequently, we obtain the inequality 
\begin{align*}
    b_{\epsilon} \left( \textrm{tru}\left( \bigoplus_{k \in A_\sigma} SH^{k}_M(K),\sigma \right) \right) \leq b_{\epsilon} \left( \textrm{tru}\left(\bigoplus_{k \in A_\sigma}HF^{\textrm{neg},k}(\sigma H),-a_{\sigma H}(\sigma)\right) \right).
\end{align*}

\vspace{0.2cm}

Let $\Delta$ be the diagonal of $(M', \omega') = ( M \times M,(-\omega)\oplus \omega)$ and $\Gamma_{\sigma} = \Gamma_{\varphi_H^{\sigma}}$ be the graph of the Hamiltonian diffeomorphism $\varphi_{\sigma H} = \varphi_{\sigma H}^1 = \varphi_H^{\sigma}$ of $\sigma H$, namely,
    \begin{align*}
        \Gamma_{\sigma} = \left\{ (x, \varphi_{\sigma H}(x)) \in M \times M \bigmid x\in M \right\}.
        \end{align*}
        Then every element of $CF(\sigma H)$ arises from an intersection point of the intersection $\Delta \cap \Gamma_{\sigma}$ of the diagonal $\Delta$ and the graph $\Gamma_\sigma$. Recall from the proof of Theorem \ref{rene}, in order to estimate $$b_{\epsilon} \left( \textrm{tru}\left(\bigoplus_{k \in A_\sigma}HF^{\textrm{neg},k}(\sigma H),-a_{\sigma H}(\sigma)\right) \right),$$
        it suffices to consider only the lower orbits of $\sigma H$. This implies that
\begin{align*}
    b_{\epsilon} \left( \textrm{tru}\left(\bigoplus_{k \in A_\sigma}HF^{\textrm{neg},k}(\sigma H),-a_{\sigma H}(\sigma)\right) \right) \leq \left| (\Delta \cap \Gamma_\sigma )\cap (K_{1+r_1} \times K_{1+r_1})  \right|.
\end{align*}

\vspace{0.2cm}

Consider the Lagrangian tomograph $\Delta_s$ of $\Delta$ as above. We may assume that the Hofer distance between $\Delta$ and $\Delta_s$ is small enough so that  
    \begin{align*}
       b_{2\epsilon} \left( \textrm{tru}\left(\bigoplus_{k \in A_\sigma}HF^{\textrm{neg},k}(\sigma H),-a_{\sigma H}(\sigma)\right) \right) \leq | (\Delta_s \cap \Gamma_{\sigma}) \cap (K_{1+r_1} \times K_{1+r_1})|.
    \end{align*}
Let $N_{\sigma}(s) = | (\Delta_s \cap \Gamma_{\sigma}) \cap (K_{1+r_1} \times K_{1+r_1})|$. Then by Crofton's inequality (Theorem \ref{crof}), we have
    \begin{align*}
        \int_B N_{\sigma}(s) ds \leq D \,\textrm{Vol}(\Gamma_{\sigma} \cap(K_{1+r_1} \times K_{1+r_1}))
    \end{align*}
    for some constant $D>0$, not depending on $\sigma$. Then there exists a constant $D'$, also independent of $\sigma$, such that
    \begin{align}\label{a1}
    b_{2\epsilon} \left( \textrm{tru}\left(\bigoplus_{k \in A_\sigma}HF^{\textrm{neg},k}(\sigma H),-a_{\sigma H}(\sigma)\right) \right) \leq D' \,\textrm{Vol}(\Gamma_{\sigma} \cap(K_{1+r_1} \times K_{1+r_1}))
    \end{align}
   By Yomdin's theorem (Theorem \ref{y}), 
   \begin{align}\label{yy}
       \limsup_{\sigma \to \infty} \frac{1}{\sigma} \log \textrm{Vol}\left( \Gamma_\sigma \cap(K_{1+r_1} \times K_{1+r_1}) \right)
      \leq h_{\textrm{vol}}(\varphi_H|_{K_{1+r_1} })
      \leq h_{\textrm{top}}( \varphi_{H}|_{K_{1+r_1} }).
   \end{align}

\vspace{0.2cm}

Combining the results from the discussion above, we obtain

    \begin{align*}
    \hbar_{2\epsilon}(SH_M(K))&= 
    \limsup_{\sigma \to \infty} \frac{1}{\sigma} \log^+ b_{2\epsilon} \left( \textrm{tru}\left(  SH_M(K),\sigma \right) \right) \\
    &= \limsup_{\sigma \to \infty} \frac{1}{\sigma} \log^+ b_{2\epsilon} \left( \textrm{tru}\left( \bigoplus_{k \in A_\sigma} SH^{k}_M(K),\sigma \right) \right) \\
    &= \limsup_{\sigma \to \infty} \frac{1}{\sigma} \log^+ b_{2\epsilon} \left( \textrm{tru}\left(\bigoplus_{k \in A_\sigma}HF^{\textrm{neg},k}(\sigma_\ell H),-a_{\sigma H}(\sigma)\right) \right) &&\textrm{by \eqref{oy}}\\
    &\leq \limsup_{\sigma \to \infty} \frac{1}{\sigma} \log \textrm{Vol}\left( \Gamma_\ell \cap(K_{1+r_1} \times K_{1+r_1}) \right)&&\textrm{by \eqref{a1}}\\
     &= h_{\textrm{top}}(\varphi_{H}|_{K_{1+r_1}}) &&\textrm{by \eqref{yy}}\\&= s_{H} h_{\textrm{top}}(\varphi_\alpha) 
    \end{align*}
    where the last identity follows from Lemma 5.12 of \cite{fls}. Taking $\epsilon \to 0$, we have
\begin{align}\label{shin}
    \hbar(SH_M(K)) \leq s_{H} h_{\textrm{top}}(\varphi_\alpha).
\end{align}

\vspace{0.2cm}

Since the inequality holds for any $K$-semi-admissible Hamiltonian function $H$ satisfying \eqref{shin}, we may define
\begin{align*}
    C'= \inf \left\{s_H \bigmid c_H \geq 1+3r_1  \right\}.
\end{align*}
As shown above $s_H >1$ whenever $H$ satisfies \eqref{ch} and therefore $C' \geq1$. In general, this number $C'$ depends on the choice of $r_1$; we therefore denote it by $C'(r_1)$. To further constrain this number, consider
\begin{align*}
    r_{\textrm{max}} = \sup \left\{r>0\bigmid K \cup\left(\partial K \times [1,1+3r]\right)\,\,\textrm{is symplectically embedded into}\,\,M \right\}.
    \end{align*}
    This number $r_{\textrm{max}}$ is finite because $M$ is assumed to be compact. We then define the constant 
    \begin{align*}
        C = \inf\left\{ C'(r) \bigmid 0<r\leq r_{\textrm{max}} \right\}.
    \end{align*}
    By construction, the constant $C$ does depend on the pair  $(M,K) $. With this choice, the inequality \eqref{shin} strengthens to
    \begin{align*}
        \hbar(SH_M(K)) \leq C h_{\textrm{top}}(\varphi_\alpha).
    \end{align*}
\end{proof}

\Addresses

\begin{thebibliography} {99}
\bibitem{a} J. Ahn, $S^1$-equivariant relative symplectic cohomology and relative symplectic capacities, preprint,  arXiv:2410.01977.

\bibitem{a2} J. Ahn, Comparison of symplectic capacities, preprint, arXiv:2504.10431.


\bibitem{cgg}  E. Cineli, V. Ginzburg and B. Gurel, Topological entropy of Hamiltonian diffeomorphisms: a persistence homology and Floer theory perspective, Mathematische Zeitschrift, 308 (2024), no. 4, Paper No. 73, 38 pp.

\bibitem{cggm} E. Cineli, V. Ginzburg, B. Gurel and M. Mazzucchelli, On the barcode entropy of Reeb flows, Selecta Mathematica. New Series,
 31 (2025), no. 4, Paper No. 64, 36 pp.

\bibitem{cggm2} E. Cineli, V. Ginzburg, B. Gurel and M. Mazzucchelli, Invariant sets and hyperbolic closed Reeb orbits, preprint, arXiv:2309.04576

\bibitem{co}  K. Cieliebak and A. Oancea, Symplectic homology and the Eilenberg–Steenrod axioms, Algebraic and Geometric Topology, 18 (2018), no. 4, 1953 – 2130.

\bibitem{csgo} F. Chazal, V. de Silva, M. Glisse and S. Oudot, The structure and stability of persistence modules, SpringerBriefs in Mathematics, Springer, [Cham], 2016. x+120 pp, ISBN:978-3-319-42543-6, ISBN:978-3-319-42545-0


\bibitem{dgpz} A. Dickstein, Y. Ganor, L. Polterovich and F. Zapolsky, Symplectic topology and ideal-valued measures, Selecta Mathematica. New Series, 30 (2024), no. 5, Paper No. 88, 92 pp.

\bibitem{f} R. Fernandes, Barcode entropy and wrapped Floer homology, prerprint, arXiv:2410.05528.

\bibitem{f2} R. Fernandes, Wrapped Floer homology and hyperbolic sets, prerprint, arXiv:2501.06654.

\bibitem{fls} E. Fender, S. Lee and B. Sohn, Barcode entropy for Reeb flows on contact manifolds with Liouville fillings, preprint, arXiv:2305.04770.

\bibitem{g} M. Gromov, Entropy, homology and semialgebraic geometry, Séminaire Bourbaki, Vol. 1985/86, Astérisque No. 145-146 (1987), 5, 225–240.



\bibitem{ggm} V. Ginzburg, B. Gurel and M. Mazzucchelli, Barcode entropy of geodesic flows, to appear in Journal of the European Mathematical Society. 

\bibitem{kh} A. Katok and B. Hasselblatt, Introduction to modern theory of dynamical systems, Encyclopedia of Mathematics and its Applications, 54, Cambridge University Press, Cambridge (1995) xviii+802 pp, ISBN:0-521-34187-6.



\bibitem{m} M. Meiwes, On the barcode entropy of Lagrangian submanifolds, preprint, arXiv:2401.07034.

\bibitem{msv} C.Y. Mak, Y. Sun and U. Varolgunes, A characterization of heaviness in terms of relative symplectic cohomology, Journal of Topology, 17 (2024), no. 1, Paper No. e12327, 26 pp.

\bibitem{p} J. Pardon, Contact homology and virtual fundamental cycles, Journal of American Mathematical Society, 32 (2019), no. 3, 825–919.

\bibitem{prsz} L. Polterovich, D. Rosen, K. Samvelyan and J. Zhang, Topological persistence in geometry and analysis, University Lecture Series, 74, American Mathematical Society, Providence, RI, (2020), ©2020. xi+128 pp, ISBN:978-1-4704-5495-1. 

\bibitem{ps} L. Polterovich, E. Shelukhin, Autonomous Hamiltonian flows, Hofer's geometry and persistence modules, Selecta Mathematica. New Series, 22 (2016), no. 1, 227 – 296.




\bibitem{sz} D. Salamon and E. Zehnder, Morse theory for periodic solutions of Hamiltonian systems and the Maslov index, Communications on Pure and Applied Mathematics, 45 (1992), 1303 – 1360.

\bibitem{sun} Y. Sun, Index bounded relative symplectic cohomology, Algebraic and Geometric Topology, 24 (2024), no. 9, 4799–4836.

\bibitem{th} F. Todd and B. Hasselblatt, Hyperbolic flows, Zurich Lectures in Advanced Mathematics, EMS Publishing House, Berlin, [2019], ©2019. xiv+723 pp, 
ISBN:978-3-03719-200-9.

\bibitem{uz} M. Usher and J. Zhang, Persistent homology and Floer-Novikov theory, Geometry and Topology, 20 (2016), no. 6, 3333–3430.

\bibitem{v} U. Varolgunes, Mayer–Vietoris property for relative symplectic cohomology, Geometry and Topology, 25 (2021), 547 - 642. 


\bibitem{y} Y. Yomdin, Volume growth and entropy, Israel Journal of Mathematics, 57 (1987), 285 - 300.




\end{thebibliography}
\end{document}